\documentclass[11pt]{amsart}

\usepackage{mathtools}
\usepackage{amsmath}
\usepackage{amsthm}
\usepackage{amssymb}
\usepackage{dsfont}
\usepackage[english]{babel}
\usepackage[utf8x]{inputenc}
\usepackage{overpic}
\usepackage{enumerate}
\usepackage{pdflscape}
\usepackage{float}
\usepackage{hyperref}
\usepackage[all,cmtip]{xy}

\theoremstyle{plain}
\newtheorem{theo}{Theorem}

\theoremstyle{definition}

\newtheoremstyle{pl}
{3pt}
{3pt}
{\itshape}
{}
{\scshape}
{.}
{.5em}
{}

\newtheoremstyle{df}
{3pt}
{3pt}
{\normalfont}
{}
{\scshape}
{.}
{.5em}
{}

\newtheoremstyle{rm}
{3pt}
{3pt}
{\normalfont}
{}
{\scshape}
{.}
{.5em}
{}

\theoremstyle{pl}
\newtheorem{thm}{Theorem}[section]			
\newtheorem{lem}[thm]{Lemma}
\newtheorem{cor}[thm]{Corollary}
\newtheorem{pro}[thm]{Proposition}

\newtheorem*{thm*}{Theorem}			
\newtheorem*{lem*}{Lemma}
\newtheorem*{cor*}{Corollary}
\newtheorem*{pro*}{Proposition}

\theoremstyle{df}
\newtheorem*{dfn}{Definition}

\theoremstyle{rm}

\newtheorem*{ex*}{Example}

\newcommand{\ep}{
\epsilon
}
\newcommand{\ol}[1]{
\overline{#1}
} 
\newcommand{\mc}[1]{
\mathcal{#1}
}
\newcommand{\mb}[1]{
\mathbb{#1}
}
\newcommand{\T}{
\mc{T}
}

\begin{document}

\title{Small eigenvalues of random 3-manifolds}
\author{Ursula Hamenst\"adt and Gabriele Viaggi}
\thanks{Both authors were partially supported by ERC grant "Moduli"}
\date{January 18, 2021}

\begin{abstract}
We show that for every $g\geq 2$ there exists a number $c=c(g)>0$ such that the smallest positive eigenvalue of a random closed 3-manifold $M$ of Heegaard genus $g$ is at most $c(g)/{\rm vol}(M)^2$. 
\end{abstract}

\maketitle

\section{Introduction}

By celebrated work of Perelman, any closed orientable aspherical atoroidal 3-manifold admits a hyperbolic metric, and such a metric is unique by Mostow rigidity. In recent years, there was considerable progress in the understanding of the relation between geometric and topological invariants of such a manifold. The program to construct an explicit combinatorial model which describes the geometry up to uniform quasi-isometry turned out to be particularly fruitful \cite{M10}, \cite{BCM12}, \cite{BMNS16}, but it is far from completed.

The main purpose of this article is obtain an understanding of  geometric and topological invariants for {\em random} hyperbolic Heegaard splittings of genus $g\ge 2$ in the sense of Dunfield and Thurston \cite{DT06}. A Heegaard splitting of genus $g\ge 2$ is a 3-manifold diffeomorphic to one of the following form: Consider two copies of a handlebody $H_g$ of genus $g$ and glue them along the boundary $\Sigma:=\partial H_g$ with an orientation reversing diffeomorphism $f:\Sigma\rightarrow\Sigma$. The resulting 3-manifold $M_f:=H_g\cup_fH_g$ only depends on the isotopy class of $f$, thus it is well defined for the representative of $f$ in the {\em mapping class group} ${\rm Mod}(\Sigma)$. Furthermore, if $f$ is sufficiently complicated in an appropriate topological sense (see Hempel \cite{He01}), then $M_f$ is aspherical and atoroidal and, hence, hyperbolic. 

Notice that, by standard 3-manifold topology, every orientable 3-manifold $M$ admits a Heegaard splitting description, that is, $M$ is diffeomorphic to a 3-manifold of the form $M_f$ obtained from the previous procedure for some $g\ge 2$ and $f\in{\rm Mod}(\Sigma)$. Thus, we have a correspondence between 3-manifolds $M$ of {\em Heegaard genus} at most $g$ and elements in ${\rm Mod}(\Sigma)$ (in fact, every such 3-manifold $M$ corresponds to a double coset in this group, see \cite{DT06}).

Now let us choose a {\em symmetric} probability measure on ${\rm Mod}(\Sigma)$ whose finite support generates the group. This measure generates a {\em random walk} on ${\rm Mod}(\Sigma)$, and hence it induces a notion of a {\em random Heegaard splitting}, or {\em random 3-manifold}, glued from two handlebodies with a {\rm random gluing map}. A random 3-manifold is hyperbolic \cite{Ma10} and hence we can study the behavior of geometric invariants of such random hyperbolic 3-manifolds $M_f$.

Our main technical result (Theorem \ref{gluing}) constructs for a suitable class of gluings $M_f$ a Riemannian metric of sectional curvature close to $-1$ everywhere and different from $-1$ only on two regions whose geometry (injectivity radius and intrinsic diameter) is uniformly controlled. The constraints that $f$ has to fulfill for the construction of such a metric are satisfied for random gluing maps.

We use this construction to obtain information on the {\em spectrum of the Laplacian} of a random hyperbolic 3-manifold. 

For every closed hyperbolic 3-manifold $M$, list the positive eigenvalues as $0<\lambda_1(M)\leq\lambda_2(M)\leq\cdots$, with each eigenvalue repeated according to its multiplicity. By \cite{S82} and \cite{H18}, there exists a universal constant $\chi>0$ such that
\[
\lambda_1(M)\geq \frac{\chi}{{\rm vol}(M)^2}\quad \text{ and }\lambda_{{\rm vol}(M)/\chi}(M)\geq \chi
\]
for every closed hyperbolic 3-manifold $M$. Manifolds which fibre over the circle provide examples for which these estimates are essentially sharp. We refer to the introduction of \cite{BGH16} for a more comprehensive discussion.

On the other hand, it follows from the work of Buser \cite{Bu82} and Lackenby \cite{L06} that there exists a number $b(g)>0$ such that for a hyperbolic 3-manifold $M$ of Heegaard genus $g$, there is a  bound
\[
\lambda_1(M)\leq\frac{b(g)}{{\rm vol}(M)}.
\]
Hyperbolic 3-manifolds constructed from expander graphs have arbitrarily large volume, yet their smallest positive eigenvalue is 
bounded from below by a universal constant. Hence in this estimate, the dependence of the constant $b(g)$ on the Heegaard genus $g$ can not be avoided. 

Under geometric constraints, one obtains better estimates. White \cite{Wh13} showed that for every $\ep>0$ there is a number $a(g,\ep)>0$ such that $\lambda_1(M)\leq a(g,\ep)/{\rm vol}(M)^2$ if $M$ has Heegaard genus $g$ and the injectivity radius of $M$ is bounded from below by $\ep$. 

A similar behaviour holds true for random hyperbolic 3-manifolds fibering over the circle, with fibre genus $g$ \cite{BGH16}. Notice that for these manifolds (and for random Heegaard splittings as well) there is no uniform lower bound on the injectivity radius. Using the model metric for random Heegaard splittings as our main tool we show:

\begin{theo}
\label{main}
For every $g\geq 2$ there exists a number $c(g)>1$ such that
\[
\lambda_1(M_f)\leq \frac{c(g)}{{\rm vol}(M_f)^2}\quad \text{ and }\lambda_{{\rm vol}(M_f)/c(g)}(M_f)\leq c(g)
\]
for a random Heegaard splitting $M_f$ of genus $g$.
\end{theo}  

Here the upper bound for $\lambda_{{\rm vol}(M_f)/c(g)}(M_f)$ is a straightforward consequence of domain monotonicity with Dirichlet boundary conditions. For the upper bound for $\lambda_1(M_f)$, we expect that the dependence of the constant $c(g)$ on $g$ can not be avoided. 

\subsection*{Strategy of the proof}
As mentioned above, our main technical result is Theorem \ref{gluing} which provides an explicit Riemannian metric of curvature close to $-1$ on $M_f$ with some constraints on the gluing map $f$. Constructions of geometrically controlled model metrics appear frequently in the literature, for example as a main tool in \cite{NS09} and in \cite{Na05}. For hyperbolic 3-manifolds diffeomorphic to $\Sigma\times\mb{R}$, there is a completely explicit combinatorial model for the geometry \cite{M10}, \cite{BCM12}. More recently, these results were used to describe explicitly the geometry of hyperbolic 3-manifolds with a lower bound on the injectivity radius and some topological constraints \cite{BMNS16}. 

We can not apply the constructions in \cite{BMNS16} as there are no lower bounds for the injectivity radius of a random hyperbolic 3-manifold $M_f$. Instead we use properties of the random walk to locate regions in a random 3-manifold which are diffeomorphic to a trivial I-bundle over a closed surface and such that a combinatorial model would predict a uniform lower bound on the injectivity radius in those regions. This is the constraint on the gluing map required in Theorem \ref{gluing}. 
The model metric is then constructed by cutting $M_f$ open at two such regions and by using information on suitable model metrics for the pieces. 

For random hyperbolic 3-manifolds $M_f$, we find that the spectrum of the model metric fulfills the properties stated in Theorem \ref{main}.

The last step consists in comparing the model metric on $M$ and the hyperbolic metric. A result of Tian \cite{Ti90} implies that, in our setting, the model metric is $\mc{C}^2$-close to a hyperbolic metric on $M$. As this work is neither published nor available in electronic form, we prove a weak substitute which is sufficient for the proof of Theorem \ref{main}. Our argument is based on the methods introduced in \cite{BCG95}. 

\subsection*{Organization of the article} 
In Section \ref{hyperbolicstructures} we review some deformation theory of convex cocompact hyperbolic handlebodies and I-bundles. The building blocks for the model metric will be pieces of these manifolds. We also collect some properties of the pointed geometric topology for hyperbolic 3-manifolds, this is one of the main tools used in the proof of Theorem \ref{gluing}.

In Section \ref{productregions} we describe a cut and glue construction that we use for building the model manifold.

In Sections \ref{relativebounded} and \ref{almostisometric} we discuss the applicability of such construction and prove Theorem \ref{gluing} on the existence and structure of the model metric. The conditions that we will have to check reduce to a relative version of bounded combinatorics for the gluing. Most importantly, such conditions imply a good control on the collar geometry of convex-cocompact hyperbolic handlebodies as described in Proposition \ref{large-thick collar}.

In Section \ref{randomheegaard} we show that random hyperbolic 3-manifolds have the properties required by Theorem \ref{gluing}, and in Section \ref{geometric} we relate the model metric to the underlying hyperbolic one using tools from \cite{BCG95}. The information on the hyperbolic metric we obtain then leads to Theorem \ref{main}.

\noindent
{\bf Acknowledgement:} We thank the referee for careful reading and for many helpful suggestions for clarification of the arguments and 
improvement of the article.

\section{Convex-cocompact handlebodies and quasi-fuchsian manifolds}\label{hyperbolicstructures}
Similar to \cite{Na05}, \cite{NS09}, \cite{BMNS16} we will build a concrete Riemannian metric on the Heegaard splitting $M_f=H_g\cup_fH_g$ by gluing together elementary building blocks. Such a metric will be {\em almost} hyperbolic, in the sense that it will have constant sectional curvature $-1$ at every point except on two regions which have small size. On those regions the curvature is contained in the interval $(-1-\ep,-1+\ep)$ where $\ep<1$ is a small constant.   

The building blocks that we are going to use are pieces of {\em convex-cocompact handlebodies} and {\em quasi-fuchsian manifolds} which are classes of hyperbolic structures on $H_g$ and $\Sigma\times[0,1]$ respectively. The goal of this section is to introduce these objects and recall some results from their deformation theory. In the next section we will explain how to cut from them the {\em gluing blocks} that we need, and how we plan to glue them together.

The applicability of the cut and glue construction is then discussed in Sections \ref{relativebounded} and \ref{almostisometric}. As a preparation, at the end of this section we recall some basic general compactness properties of the geometric topology which is one of the main tools that we will use.

\subsection{Kleinian groups}
We start by recalling some general terminology about Kleinian groups, that is, discrete subgroups $\Gamma<{\rm Isom}^+(\mb{H}^3)$. We always assume that $\Gamma$ is torsion free. Each such group has an associated {\em limit set} $\Lambda_\Gamma\subset\partial\mb{H}^3$, which consists of the points at infinity of a $\Gamma$-orbit closure, and a {\em domain of discontinuity} $\Omega_\Gamma:=\partial\mb{H}^3-\Lambda_\Gamma$. The group $\Gamma$ acts freely and properly discontinuously $\mb{H}^3\cup\Omega_\Gamma$ and the quotient $\mb{H}^3\cup\Omega_\Gamma/\Gamma$ is a 3-manifold with boundary $\Omega_\Gamma/\Gamma$. If $\Gamma$ is not abelian, then $\Omega_\Gamma$ has a natural complete hyperbolic metric, called the {\em Poincar\'e metric}, which is preserved by $\Gamma$. Thus, the boundary surface $\Omega_\Gamma/\Gamma$ inherhits a natural hyperbolic structure. By Ahlfors' Finiteness Theorem \cite{A64}, if $\Gamma$ is torsion free, non-abelian and finitely generated, then $\Omega_\Gamma/\Gamma$ has finite area.
 
\subsection{Teichm\"uller space and mapping class group} 
The Kleinian groups that we are going to consider are naturally parametrized by points in the Teichm\"uller space $\T$ of marked hyperbolic structures on a closed orientable surface of genus $g\ge 2$. 

For sufficiently small $\delta>0$ we denote by $\T_\delta\subset\T$ the subset of Teichm\"uller space consisting of those (marked) hyperbolic metrics on $\Sigma$ with injectivity radius at least $\delta$. We will extensively use a classical theorem of Mumford (see Theorem 12.6 of \cite{FM}) that the mapping class group ${\rm Mod}(\Sigma)$ acts cocompactly on each $\T_\delta$.

\subsection{Convex-cocompact handlebodies}
We now describe a class of hyperbolic 3-manifolds homeomorphic to handlebodies.

\begin{center}
\label{notation1}
\begin{minipage}{.8\linewidth}
{\bf Standing assumptions}. Fix once and for all a genus $g\ge 2$. Let $H_g$ be a handlebody of genus $g$ with boundary surface $\Sigma:=\partial H_g$. We fix on $H_g$ an orientation, and we coherently orient $\Sigma$ as the boundary of $H_g$.
\end{minipage}
\end{center}

For the material in this section we mainly refer to Chapter 7 of \cite{CM}. 

\begin{dfn}[Convex-Cocompact Handlebody]
A {\em convex cocompact marked} hyperbolic structure on the handlebody $H_g$ is a quotient $N=\mb{H}^3/\Gamma$ of the hyperbolic 3-space by a discrete free subgroup $\Gamma\simeq\mb{F}_g$ together with an orientation preserving homeomorphism $\phi:H_g\rightarrow{\hat N}:=\mb{H}^3\cup\Omega_\Gamma/\Gamma$ (the {\em marking}). We say that the marked structures $\phi:H_g\rightarrow{\hat N}$ and $\phi':H_g\rightarrow{\hat N}'$ are {\em equivalent} if there exists an orientation preserving homeomorphism $f:{\hat N}\rightarrow{\hat N}'$ that restricts to an isometry $f:N\rightarrow N'$ and such that $f\phi$ is {\em isotopic} to $\phi'$.
\end{dfn}

Notice that the boundary $\partial{\hat N}:=\Omega_\Gamma/\Gamma$ comes equipped with a marking $\phi:\Sigma=\partial H_g\rightarrow\partial{\hat N}$ and a hyperbolic metric. This determines a point in the Teichm\"uller space $\T$ of the boundary $\Sigma=\partial H_g$ which is called the {\em conformal boundary} of $N$.

By classical results due to Bers \cite{Bers70}, Kra \cite{Kra72}, Maskit \cite{Maskit71}, equivalence classes of convex-cocompact marked handlebodies are parametrized by the Teichm\"uller space $\T$ via the map that associates to the structure its conformal boundary (see Chapter 7 of \cite{CM}, in particular Theorem 7.2.9). Given $X\in\T$ we denote by $H(X)$ the convex cocompact handlebody with conformal boundary $X$.

\subsection{Quasi-fuchsian manifolds} 
We consider now a class hyperbolic structures on the topological model $\Sigma\times[0,1]$. Again, we mainly refer to Chapter 7 of \cite{CM} for the material presented here. 

\begin{dfn}[Quasi-Fuchsian Manifolds]
A {\em convex cocompact marked} hyperbolic structure on $\Sigma\times[0,1]$, also called {\em quasi-fuchsian} manifold, is a quotient 
$Q=\mb{H}^3/\Gamma$ of the hyperbolic 3-space $\mb{H}^3$ by a discrete surface subgroup $\Gamma\simeq\pi_1(\Sigma)$ together with an orientation preserving homeomorphism $\phi:\Sigma\times[0,1]\rightarrow{\hat Q}:=\mb{H}^3\cup\Omega_\Gamma/\Gamma$ (the {\em marking}). As before, two quasi-fuchsian manifolds $Q,Q'$ are equivalent if they differ by an orientation preserving homeomorphism $f:{\hat Q}\rightarrow{\hat Q}'$ which resctricts to an isometry $Q\rightarrow Q'$ and is isotopic to the identity (with respect to the markings). 
\end{dfn}

The conformal boundary $\partial{\hat Q}=\Omega_\Gamma/\Gamma$ has now two connected components both homeomorphic to $\Sigma$, but with opposite orientations. The restrictions of the marking $\phi$ to $\Sigma\times\{0\}\cup\Sigma\times\{1\}$ together with the intrinsic hyperbolic structure on $\Omega_\Gamma/\Gamma$ determine a pair of points in Teichm\"uller space.

By Bers' Simultaneous Uniformization \cite{Bers60}, equivalence classes of quasi-fuchsian manifolds $Q$ are parametrized by $\T\times\T$ via the map that associates to $Q$ the conformal boundary $\partial{\hat Q}$. Given a pair $(Y,X)\in\T\times\T$ we denote by $Q(Y,X)$ the unique quasi-fuchsian manifold that realizes those boundary data $Y$ and $X$ on $\Sigma\times\{0\}$ and $\Sigma\times\{1\}$ respectively.

The mapping class group ${\rm Mod}(\Sigma)$ acts on the space of quasi-fuchsian manifolds by precomposition of marking, that is $\phi\in{\rm Mod}(\Sigma)$ acts as $\phi^{-1}\times\mb{I}$ on $\Sigma\times[0,1]$. On the the conformal boundary, the action coincides with the diagonal action ${\rm Mod}(\Sigma)\curvearrowright\T\times\T$.

\subsection{Geometry and topology of the convex core}
Convex-cocompact hyperbolic structures $N$ and $Q$ on $H_g$ and $\Sigma\times[0,1]$ are infinite volume Riemannian manifolds, but we will only use their {\em convex cores} which are {\em compact submanifolds} that are {\em convex} in the sense that they contain all the geodesics joining two of their points. We now describe some of the topological and geometric features of these cores.

In general, for every torsion free Kleinian group $\Gamma$ we can always construct the convex hull of the limit set $\mc{CH}(\Lambda_\Gamma)$. This is a convex subset of $\mb{H}^3$ invariant under $\Gamma$. The quotient $\mc{CH}(\Lambda_\Gamma)/\Gamma$ is called the {\em convex core} of the hyperbolic 3-manifold $M:=\mb{H}^3/\Gamma$ and is denoted by $\mc{CC}(M)$.
      
It is a standard fact that, for convex cocompact hyperbolic structures $N$ and $Q$ on $H_g$ and $\Sigma\times[0,1]$, the convex cores $\mc{CC}(N)$ and $\mc{CC}(Q)$ are compact topological codimension 0 (except in the {\em fuchsian} case which we ignore) submanifolds whose boundary is {\em parallel} to the boundary of ${\hat N}$ and ${\hat Q}$. This property provides homeomorphisms $H_g\simeq\mc{CC}(N)$ and $\Sigma\times[0,1]\simeq\mc{CC}(Q)$ isotopic to the markings $H_g\simeq{\hat N}$ and $\Sigma\times[0,1]\simeq{\hat Q}$.

From a geometric point of view, the boundary of the convex core $\partial\mc{CC}(M)$ always has the structure of an embedded convex {\em pleated surface} (see Chapter I.5 of \cite{CEG06}). In particular, it always has an intrinsic hyperbolic metric. For convex cocompact hyperbolic structures on $H_g$ and $\Sigma\times[0,1]$, the boundary of the convex core comes naturally equipped with a marking, that is, an identification with $\Sigma$. Using it we can describe each component of the boundary of the convex core as a point in Teichm\"uller space $\T$. 

By work of Sullivan (see Chapter II.2 of \cite{CEG06}) and Bridgeman-Canary \cite{BC03}, the geometry of the conformal boundary and the geometry of the convex core are strictly tied: There is a natural {\em nearest pont retraction} $\partial{\hat M}\rightarrow\partial\mc{CC}(M)$ which has the following properties:

\begin{thm}[Bridgeman-Canary, \cite{BC03}]
\label{bridgeman-canary}
There are maps  $J,G:(0,\infty)\to(1,\infty)$ such that the following holds: Let $\Gamma<{\rm Isom}^+(\mb{H}^3)$ be a finitely generated, non-abelian, torsion free Kleinian group. Suppose that the length, measured with respect to the Poincar\'e metric, of every curve in the conformal boundary $\Omega_\Gamma/\Gamma$ which is compressible in the 3-manifold $\mb{H}^3\cup\Omega_\Gamma/\Gamma$ is bounded from below by $\delta>0$. Then, the natural nearest point retraction from the conformal boundary to the boundary of the convex core is $J(\delta)$-Lipschitz and admits a $G(\delta)$-Lipschitz homotopy inverse.  
\end{thm}

We notice that, if a component $X$ of $\partial{\hat M}$ is contained in $\T_\delta$, then, by Theorem \ref{bridgeman-canary}, the corresponding component $\partial_X\mc{CC}(M)$ of $\partial\mc{CC}(M)$ lies in $\T_{2\delta/G(\delta)}$. 

Theorem \ref{bridgeman-canary} has the following immediate consequence:

\begin{lem}
\label{thick near boundary}
For every $\delta>0$ and $g\ge 2$ there exists $\eta=\eta(\delta,g)>0$ such that the following holds: Let $M$ be either a marked convex cocompact hyperbolic structure on $H_g$ or on $\Sigma\times[0,1]$. If each component of the conformal boundary $\partial{\hat M}$ is contained in $\T_\delta$, then $\inf_{x\in\partial\mc{CC}(M)}\{{\rm inj}_x(M)\}\ge\eta$.
\end{lem}

\begin{proof}
Notice that, by Theorem \ref{bridgeman-canary}, each component $C$ of $\partial\mc{CC}(M)$ lies in $\T_{2\delta/G(\delta)}$. Therefore, as ${\rm Mod}(\Sigma)$ acts cocompactly on $\T_{2\delta/G(\delta)}$ (see Theorem 12.6 of \cite{FM}), there is a uniform upper bound $L=L(g,\delta)>0$ on the intrinsic diameter of each component $C$ of $\partial\mc{CC}(M)$. 

Let $\eta_0>0$ be a Margulis constant for hyperbolic 3-manifolds (see Chapter D of \cite{BP92}). Pick $x\in C$. If ${\rm inj}_x(M)<\eta<\eta_0$, then $x$ lies inside a $\eta_0$-Margulis tube $x\in\mb{T}$ at a distance from the boundary $\partial\mb{T}$ of coarsely $\log(\eta_0/\eta)$ (see \cite{BM82}). If $\eta$ is very small, then $C$, having uniformly bounded diameter, would be contained in $\mb{T}$, but this is absurd as the inclusion of $\pi_1(C,x)$ in $\pi_1(M,x)$ is surjective.
\end{proof}

\subsection{Convergence of hyperbolic manifolds} 
We conclude this section with a discussion of geometric convergence of hyperbolic manifolds. This is one of the main tools in the proofs of our main results in Sections \ref{relativebounded} and \ref{almostisometric}. 

Let $(M,\rho_M)$ be a complete hyperbolic surface or 3-manifold where $\rho_M$ denotes the Riemannian metric. Define the following:

\begin{dfn}[$\mc{C}^2$-norm]
Let $U\subset M$ be an open subset. Denote by $\nabla$ the Levi-Civita connection on $(M,\rho_M)$. Denote by $|\bullet|_x$ the norm induced by the inner product $\rho_M(x)$ of $T_xU$ on tensors on $T_xU$. Let $\tau$ be a tensor field on $U$. The $\mc{C}^2$-norm of $\tau$ at $x$ is the quantity
\[
\left|\left|\tau\right|\right|_{\mc{C}^2,x}:=\left|\tau(x)\right|_x+\left|\nabla\tau(x)\right|_x+\left|\nabla^2\tau(x)\right|_x.
\]
Similarly, the $\mc{C}^2$-norm of $\tau$ on $U$ is given by
\[
\left|\left|\tau\right|\right|_{\mc{C}^2(U)}:=\sup_{x\in U}\left|\tau(x)\right|_x+\sup_{x\in U}\left|\nabla\tau(x)\right|_x+\sup_{x\in U}\left|\nabla^2\tau(x)\right|_x.
\]
\end{dfn}

Using this norm we define the following:

\begin{dfn}[Geometric Convergence]\label{geomconv}
A sequence $\{(M_n,m_n)\}_{n\in\mb{N}}$ of hyperbolic surfaces or 3-manifolds with basepoints is said to converge in the {\em pointed geometric topology} to a pointed hyperbolic surface or 3-manifold $(M,m)$ if the following conditions are satisfied: For every $R>0,\xi>0$ there are numbers $n(R,\xi)>0$, and for every $n\ge n(R,\xi)$ there exists a smooth embedding (the {\em approximating map}) $k_n:U_n\subset M\rightarrow M_n$ such that $k_n$ is defined on the ball $B_M(m,R)\subset U_n$ of radius $R$ centered at $m\in M$, it sends $k_n(m)=m_n$, and the restriction of $k_n$ to $B_M(m,R)$ satisfies $||\rho_M-k_n^*\rho_{M_n}||_{\mc{C}^2(B_M(m,R))}<\xi$. In this case we say that the restriction of $k_n$ to $B(m,R)$ is $\xi$-{\em almost isometric}.
\end{dfn}

For more on geometric convergence we refer to Chapter E of \cite{BP92}. We will mainly exploit the following compactness result for the pointed geometric topology (see Theorem E.1.10 of \cite{BP92}).

\begin{thm}
\label{compactness}
Let $\{(M_n,m_n)\}_{n\in\mb{N}}$ be either a sequence of pointed hyperbolic surfaces or 3-manifolds. Suppose that ${\rm inj}_{m_n}M_n\ge\eta$ for all $n\in\mb{N}$ for some positive $\eta>0$. Then there exists a subsequence that converges in the pointed geometric topology to a pointed hyperbolic surface or 3-manifold $(M,m)$.
\end{thm}

Consider the following setup: Let
\[
\{f_n:(X_n,x_n)\rightarrow (M_n,m_n)\}_{n\in\mb{N}}
\] 
be a sequence of basepoint preserving 1-Lipschitz maps $f_n$ from $\delta$-thick hyperbolic surfaces $X_n$ homeomorphic to $\Sigma$ to hyperbolic 3-manifolds $M_n$ with ${\rm inj}_{m_n}M_n\ge\eta$. Then, using Theorem \ref{compactness} and Ascoli-Arzel\`a, up to subsequences, we can assume that
\begin{itemize}
\item{$(M_n,m_n)$ converges in the pointed geometric topology to a pointed hyperbolic 3-manifold $(M,m)$.}
\item{$(X_n,x_n)$ converges in the pointed geometric topology to a pointed hyperbolic surface $(X,x)$. Since ${\rm inj}(X_n)\ge\delta$ we also have that ${\rm inj}(X)\ge\delta$ and $X$ is homeomorphic to $\Sigma$ (see Theorem 12.6 of \cite{FM}).}
\item{$f_n$ converges to a basepoint preserving 1-Lipschitz map $f:(X,x)\rightarrow(M,m)$.}
\item{The diagram 
 \[
 \xymatrix{
 (X_n,x_n)\ar[r]^{f_n} &(M_n,m_n)\\
 (X,x)\ar[r]_{f}\ar[u]^{\phi_n} &(M,m),\ar[u]_{k_n}
 }
 \]
where the vertical arrows are the approximating maps provided by the geometric convergence, commutes up homotopies that respect the basepoints and take place in small neighbourhoods of the images of $f_n\phi_n$ and $k_nf$.}
\end{itemize}

We now return to our specific setting and consider the case where $M_n$ are either (marked) convex cocompact structures on $H_g$ or on $\Sigma\times[0,1]$. For each such structure, we choose a basepoint on the boundary of the convex core $m_n\in\partial\mc{CC}(M_n)$ (which is on the component marked by $\Sigma\times\{1\}$ in the case of a quasi-fuchsian manifold, below we denote it by $\partial_1\mc{CC}(Q_n)$). Suppose that each component of the conformal boundary $\partial{\hat M}_n$ lies in $\T_\delta$. Then, by Theorem \ref{bridgeman-canary} and Lemma \ref{thick near boundary}, we have that each component of $\partial\mc{CC}(M_n)$ lies in $\T_{2\delta/G(\delta)}$ and ${\rm inj}_xM\ge\eta$ for every $x\in\partial\mc{CC}(M_n)$. This enables us to take geometric limits in the setup described above for sequences $\{\partial\mc{CC}(N_n)\subset N_n\}_{n\in\mb{N}}$ or $\{\partial_1\mc{CC}(Q_n)\subset Q_n\}_{n\in\mb{N}}$ where $N_n$ and $Q_n$ are convex-cocompact handlebodies and quasi-fuchsian manifolds.

\section{Cut and Glue construction}
\label{productregions}

We now describe a procedure to construct a Riemannian metric on a manifold diffeomorphic to the Heegaard splittings $M_f$ by gluing together the convex cores $\mc{CC}(N_1),\mc{CC}(N_2),\mc{CC}(Q)$ of two suitable convex-cocompact handlebodies $N_1,N_2$ and a quasi-fuchsian manifold $Q$ along suitable identifications $k_j:V_j\subset\mc{CC}(Q)\rightarrow U_j\subset\mc{CC}(N_j)$ of collars of their boundary components:
\[
X_f=\mc{CC}(N_1)\cup_{k_1:V_1\rightarrow U_1}\mc{CC}(Q)\cup_{k_2:V_2\rightarrow U_2}\mc{CC}(N_2).
\]

Depending on the amount of control that we want to have on the curvature of the resulting geometric Heegaard splitting $X_f\simeq M_f$, we will not glue directly the entire convex cores $\mc{CC}(N_1),\mc{CC}(Q)$ and $\mc{CC}(N_2)$, but rather, we will cut from them smaller submanifolds $N_0^1\subset\mc{CC}(N_1), Q_0\subset\mc{CC}(Q)$ and $N_0^2\subset\mc{CC}(N_2)$ for which we have a better control on the collar geometry.

We begin with the following definition:

\begin{dfn}[Product Regions]
A {\em product region in a quasi-fuchsian manifold} $Q$ is a codimension 0 submanifold $U\subset\mc{CC}(Q)$ which is  homeomorphic to $\Sigma\times[0,1]$ and whose inclusion $U\subset\mc{CC}(Q)$ is a homotopy equivalence.

A {\em product region in a convex cocompact handlebody} $N$ is a codimension 0 submanifold $U\subset\mc{CC}(N)$ contained in a topological collar of the boundary of the convex core $\partial\mc{CC}(N)$ and such that $U$ is homeomorphic to $\Sigma\times[0,1]$ and the inclusion of $U$ in the collar of $\partial\mc{CC}(N)$ is a homotopy equivalence. 
\end{dfn}

In both cases a product region comes naturally with a {\em marking} $j:\Sigma\rightarrow U$ obtained by isotopy from the one of the boundary of the convex core. This allows us to define the {\em homotopy class} of any orientation preserving diffeomorphism $k:V\rightarrow U$ between product regions: Let $j_U,j_V:\Sigma\rightarrow U,V$ be the markings of $U,V$. We have identifications
\[
\pi_1(\Sigma)\simeq_{j_U}\pi_1(U)\simeq_k\pi_1(V)\simeq_{j_V}\pi_1(\Sigma).
\]
The composition is a well defined element of ${\rm Out}^+(\pi_1(\Sigma))$ which identifies with ${\rm Mod}(\Sigma)$ by the Dehn-Nielsen-Baer Theorem (see Theorem 8.1 in \cite{FM}).

We also have the following topological useful features of product regions: By standard 3-manifold topology \cite{W68}, an embedded subsurface of $\Sigma\times[0,1]$ which is $\pi_1$-injective is parallel to $\Sigma\times\{1\}$. This implies that in both cases a product region is always parallel to the boundary of the convex core. 

Since product regions are always separating, it makes sense to say that a point $x\in N,Q$ lies above or below of $U\subset\mc{CC}(N),\mc{CC}(Q)$. Using the product structure, we can define a {\em top boundary} $\partial^+U$ and a {\em bottom boundary} $\partial^-U$. 

We will be interested in essentially two parameters of a product region $U\subset\mc{CC}(M)$ where $M$ is either a convex cocompact handlebody or a quasi-fuchsian manifold: 
\begin{itemize}
\item{The {\em diameter}, defined by $\text{\rm diam}(U):=\sup\{d_U(x,y)\left|x,y\in U\right.\}$.}
\item{The {\em width} $\text{\rm width}(U):=\inf\{d_M(x,y)\left|x\in\partial^+U,y\in\partial^-U\right.\}$.}
\end{itemize}

When the width is at least $D$ and the diameter is at most $2D$ we say that the product region has {\em size} $D$. Notice that the diameter is computed with respect to the intrinsic path metric of $U$, therefore the diameter of $U$ as a subset of $M$ is bounded from above by $\text{\rm diam}(U)$. We observe the following:

\begin{lem}
\label{inj prod region}
There exists $\ep=\ep(D)>0$ such that for each product region $U\subset\mc{CC}(M)$ of size $D$ we have $\text{\rm inj}(U):=\inf\{{\rm inj}_xM\left|x\in U\right.\}\ge\ep$.  
\end{lem}

\begin{proof}
A product region is always $\pi_1$-surjective. Having diameter bounded by $2D$, a product region cannot enter too deeply inside a Margulis tube $\mb{T}$ with large radius. Otherwise the surjective map $\pi_1(U)\rightarrow\pi_1(M)$ would factor through $\pi_1(U)\rightarrow\pi_1(\mb{T})$. 
\end{proof}

We think of product regions as the collars of the boundaries of submanifolds $N_0^1\subset\mc{CC}(N_1),N_0^2\subset\mc{CC}(N_2),Q_0\subset\mc{CC}(Q)$ where we will perform the gluing. We call the manifolds $N_0^1,N_0^2,Q_0$ gluing blocks: 

\begin{dfn}[Gluing Blocks]
Let $U\subset\mc{CC}(N)$ be a product region in a convex cocompact handlebody $N$. The {\em gluing block} associated to $U$ is the compact submanifold $N_0\subset N$ bounded by $\partial^+U$. Since $\partial N_0$ is parallel to $\partial\mc{CC}(N)$, we have $N_0\simeq H_g$.

Let $V_1,V_2\subset\mc{CC}(Q)$ be disjoint product regions in a quasi-fuchsian manifold $Q$ such that $V_2$ lies above $V_1$. The {\em gluing block} associated to $V_1,V_2$ is the compact submanifold $Q_0\subset Q$ bounded by $\partial^-V_1$ and $\partial^+V_2$. Since $\partial Q_0$ is parallel to $\partial\mc{CC}(Q)$, we have that $Q_0\simeq\Sigma\times[0,1]$. Notice that $V_2$ and $V_1$ are collars of the boundary of $Q_0$, we call them the {\em top} and {\em bottom collars} of $Q_0$. 
\end{dfn}

After discussing the elementary building blocks involved in the gluing procedure, we now describe the identifications $k_j:V_j\subset Q_0\rightarrow U_j\subset N_0^j$ of their collar product regions. 

\begin{dfn}[Almost Isometric]
Let $k:V\rightarrow U$ be an orientation preserving diffeomorphism between product regions $V\subset\mc{CC}(Q)$ and $U\subset\mc{CC}(N)$. For $\xi>0$ we say that $k$ is {\em $\xi$-almost isometric} if $\left|\left|\rho_V-k^*\rho_U\right|\right|_{\mc{C}^2(V)}<\xi$.
\end{dfn}

The identifications $k_1,k_2$ that we are going to use will be $\xi$-{\em almost isometric} orientation preserving diffeomorphisms between the collar product regions, and we will require that $k_1$ is in the homotopy class of the identity and $k_2$ is in the homotopy class of the mapping class $f\in{\rm Mod}(\Sigma)$. Using such identifications, we can form the 3-manifold
\[
X_f:=N_0^1\cup_{k_1:V_1\rightarrow U_1}Q_0\cup_{k_2:V_2\rightarrow U_2}N_0^2.
\] 

By the assumptions on the homotopy classes of $k_1$ and $k_2$, $X_f$ is diffeomorphic to $M_f$.

Now that we have a topological model $X_f$, in order to promote it to a Riemannian model, we only have to discuss how to endow it with a Riemannian metric. We do so by taking {\em convex combinations} $\theta_j\rho_Q+(1-\theta_j)k_j^*\rho_{N_j}$ on the gluing regions $V_j\subset Q_0$ with respect to some smooth {\em bump functions} $\theta_j:V_j\rightarrow[0,1]$. The main observation here is that on a product region with uniformly bounded size there is always a uniform bump function:

\begin{lem}
\label{bump function}
For all $D>0$ there exists $K>0$ such that the following holds: Let $U\simeq\Sigma\times[0,1]$ be a product region with ${\rm diam}(U)\le 2D$, ${\rm width}(U)\ge D$. Then there exists a smooth function $\theta:U\rightarrow[0,1]$ with the following properties:
\begin{itemize}
\item{Near the boundaries it is constant: $\theta|_{\partial_-U}\equiv 0$ and $\theta|_{\partial_+U}\equiv 1$.}
\item{Uniformly bounded $\mc{C}^2$-norm: $\left|\left|\theta\right|\right|_{\mc{C}^2(U)}\le K$.}
\end{itemize}
\end{lem}

\begin{proof}
We argue by contradiction. Let $U_n\subset\mc{CC}(M_n)$ be a sequence of product regions of size $D$ but without a uniform bump fuction. 

Pick a basepoint $x_n\in U_n$. By Lemma \ref{inj prod region}, we have ${\rm inj}_{x_n}(M_n)\ge\ep$ where $\ep=\ep(D)>0$ is a uniform constant. Thus, by Theorem \ref{compactness}, up to passing to subsequences, we can take a geometric limit $(M_n,x_n)\rightarrow(M,x)$. Since $U_n\subset B(x_n,2D)$, for every large enough $n$ we have $\xi$-almost isometric embeddings $k_n:U_n\rightarrow B(x,3D)$. 

Consider the closed sets $A_n=k_n(\partial^-U_n),B_n=k_n(\partial^+U_n)\subset B(x,3D)$. Since ${\rm width}(U_n)\ge D$, they have $d_M(A_n,B_n)\ge D/2$. Up to subsequences, they converge in the Hausdorff topology on closed subsets of $B(x,3D)$ to disjoint closed subsets $A,B$. 

Fix $r>0$ much smaller than $D/2$. Then there is a smooth function $\theta:B(x,3D)\rightarrow[0,1]$ which is 0,1 on the $r$-neighbourhoods $N_r(A),N_r(B)$. Considering $\theta_n=\theta k_n$ gives a smooth bump function on $U_n$ which is 0,1 on neighbourhoods of $\partial^-U_n,\partial^+U_n$ respectively and has uniformly bounded $\mc{C}^2$-norm as $k_n$ is $\xi$-almost isometric. This contradicts the initial assumptions and finishes the proof.
\end{proof}

We now summarize the cut and glue construction in the following lemma:

\begin{lem}
\label{cut-and-glue}
Let $\xi\in(0,1)$ be a small almost isometric parameter and $K>0$ be a $\mc{C}^2$-norm parameter for bump functions. Let $f$ be a mapping class. Let $N_1,N_2$ be convex cocompact handlebodies. Let $Q$ be a quasi-fuchsian manifold. Suppose that we have:
\begin{enumerate}
\item{A pair of handlebody gluing blocks $N_0^1\subset\mc{CC}(N_1),N_0^2\subset\mc{CC}(N_2)$ bounded by product regions $U_1$ and $U_2$.}
\item{An I-bundle gluing block $Q_0\subset\mc{CC}(Q)$ with bottom and top product region collars $V_1$ and $V_2$.}
\item{Orientation preserving diffeomorphisms $k_j:V_j\rightarrow U_j$ for $j=1,2$ with $k_1$ in the homotopy class of the identity and $k_2$ is in the homotopy class of $f$}
\item{Bump functions $\theta_j:V_j\rightarrow[0,1]$ for $j=1,2$ with $\theta_j\equiv 0$ in a small neighbourhood of $\partial^-V_j$ and $\theta_j\equiv 1$ in a small neighbourhood of $\partial^+V_j$.}
\end{enumerate}
Then we can form the 3-manifold
\[
X_f=N_0^1\cup_{k_1:V_1\rightarrow U_1}Q_0\cup\cup_{k_2:V_2\rightarrow U_2}N_0^2
\]  
and endow it with the Riemannian metric
\[
\rho:=\left\{
\begin{array}{l l}
\rho_{N_1} &\text{\rm on $N_0^1-U_1$},\\
\theta_1\rho_Q+(1-\theta_1)k_1^*\rho_{N_1} &\text{\rm on $V_1$},\\
\rho_Q &\text{\rm on $Q_0-(V_1\cup V_2)$},\\
(1-\theta_2)\rho_Q+\theta_2k_2^*\rho_{N_2} &\text{\rm on $V_2$},\\
\rho_{N_2} &\text{\rm on $N_0^2-U_2$}.\\
\end{array}
\right.
\]
Topologically, $X_f$ is diffeomorphic to the Heegaard splitting $M_f$ determined by $f$. Geometrically, we have the following: If $||\theta_j||_{\mc{C}^2(V_j)}<K$ for $j=1,2$ and $k_j$ is $\xi$-almost isometric with $\xi$ small enough compared to $K$, then $(X_f,\rho)$ has sectional curvature
\[
\left|1+\text{\rm sec}_{X_f}\right|<c_3K\xi
\]
where $c_3>0$ is a universal constant.
\end{lem}

In the next three sections we show how to find convex-cocompact handlebodies $N_1,N_2$ and quasi-fuchsian manifolds $Q$ that satisfy the assumptions of Lemma \ref{cut-and-glue}. The conditions that we will find are variations of earlier work of Brock, Minsky, Namazi and Souto \cite{Na05}, \cite{NS09} and \cite{BMNS16}. Our setup, however, is different from the ones of those papers as we do not assume any global control of the injectivity radius of $N_1,N_2,Q$, that is, very short curves will appear in our manifolds.

\section{Collar geometry of convex-cocompact handlebodies}
\label{relativebounded}

The goal of this section is to find conditions on $Y,X\in\T$ such that a large collar of $\partial\mc{CC}(H(X))$ closely resembles a large collar of the component of $\partial\mc{CC}(Q(Y,X))$ facing the conformal boundary $X$. This will be our main tool to single out from convex cocompact handlebodies the gluing blocks $N_0^1,N_0^2$ needed for the cut and glue construction Lemma \ref{cut-and-glue}. The following is the main result: 

\begin{pro}
\label{large-thick collar}
Let $g\ge 2$ be fixed. For all $L,\delta,\xi>0$ there exists $h=h(L,\delta,\xi)>0$ such that the following holds: If the pair $(Y,X)\in\T_\delta\times\T_\delta$ has {\em relative $\delta$-bounded combinatorics} with respect to $H_g$ and {\em height} at least $h$, then the boundary of the convex core of $N=H(X)$ has a collar of width at least $L$ which is $\xi$-almost isometric to a collar about the boundary component facing $X$ of the convex core of the quasi-fuchsian manifold $Q=Q(Y,X)$.
\end{pro}

The proof of Proposition \ref{large-thick collar} will be only carried out at the end of the section. We begin, instead, with the definition and discussion of the condition of {\em relative $\delta$-bounded combinatorics} and large {\em height} which is a variation of the ones described in \cite{Na05}, \cite{NS09} and \cite{BMNS16}. The deep connection between bounded combinatorics and the geometry of quasi-fuchsian manifolds was discovered originally by Minsky \cite{M01}, \cite{M10}. However, the Teichm\"uller perspective that we adopt here is closer to the work of Rafi \cite{R05} (see Theorem \ref{bounded} in the next section).

In order to describe what we mean by relative bounded combinatorics we briefly recall some facts about the {\em curve graph} and the {\em disk graph} and their relation with Teichm\"uller space.

\subsection{Curve graph} 
The {\em curve graph} of $\Sigma$ is the graph $\mc{C}$ whose vertices are isotopy classes of essential simple closed curves on $\Sigma$ and where two such curves are connected by an edge of length one if and only if they can be realized disjointly. 

Masur and Minsky proved in \cite{MM99} that this graph is a {\em Gromov hyperbolic} space of infinite diameter, and Klarreich \cite{K99} identified the {\em Gromov boundary} $\partial\mc{C}$ with the space of {\em filling unmeasured laminations} 
(see also \cite{H06} for a different approach).

Convergence to the boundary is governed by the {\em Gromov product} (see Section 3 of Chapter III.H of \cite{BH99}):

\begin{dfn}[Gromov Product and Convergence]
Given $\alpha,\beta,\gamma\in\mc{C}$, the quantity 
\[
(\alpha|\beta)_\gamma:=\frac{1}{2}[d_\mc{C}(\alpha,\gamma)+d_\mc{C}(\beta,\gamma)-d_\mc{C}(\alpha,\beta)]
\]
is the Gromov product of $\alpha,\beta$ based at $\gamma$. A sequence $\{\alpha_n\}_{n\in\mb{N}}\subset\mc{C}$ {\em converges at infinity} to a point in $\partial\mc{C}$ if and only if for some base point $\gamma$ (and hence for any) we have $\liminf_{n,m\rightarrow\infty}(\alpha_n|\alpha_m)_\gamma\rightarrow\infty$. If $\{\alpha_n\}_{n\in\mb{N}}$ converges at infinity and $\{\beta_n\}_{n\in\mb{N}}$ satisfies $\liminf_{n,m\to\infty}(\alpha_n|\beta_m)=\infty$, then $\{\beta_n\}_{n\in\mb{N}}$ converges to the same point in $\partial\mc{C}$.  
\end{dfn}

The geometry of the curve graph is coarsely tied to the geometry of Teichm\"uller space: There is a (coarsely well-defined) ${\rm Mod}(\Sigma)$-equivariant Lipschitz map $\Upsilon:\T\to\mc{C}$, called the {\em systole map}, that associates to every marked hyperbolic structure $X\in\T$ a shortest geodesic $\Upsilon(X)$ on it. It follows from Masur-Minsky \cite{MM99} that there exist constants $L,C>0$ only depending on $\Sigma$ such that for every Teichm\"uller geodesic $l:I\to\T$ (here $I$ can be an interval, a half-line or the whole real line) the composition $\Upsilon l:I\to\mc{C}$ is an {\em unparametrized} $(L,C)$-{\em quasi-geodesic}. Moreover, if we restrict our attention to the $\delta$-thick part $\T_\delta$ of Teichm\"uller space, then the situation improves: In \cite{H10} it is shown that for every $\delta>0$ there exist $L_\delta,C_\delta>0$ such that if $l$ is parameterized by arc length on an interval of length $l(I)\ge L_\delta$ and if $l(I)\subset \T_\delta$ then $\Upsilon l$ is a {\em parametrized} $(L_\delta,C_\delta)$-quasi-geodesic.

\subsection{Disk graph}
Since for us $\Sigma=\partial H_g$ is the boundary of the handlebody $H_g$, we can associate to it also a {\em disk graph} $\mc{D}$ which is the subgraph of $\mc{C}$ spanned by disk-bounding curves. Masur and Minsky showed in \cite{MM95} that the disk graph $\mc{D}$ is a {\em quasi-convex} subset of the curve graph $\mc{C}$. Being quasi-convex, by hyperbolicity of $\mc{C}$, there is a coarsely defined nearest point projection $\pi_{\mc{D}}:\mc{C}\rightarrow\mc{D}$.

\subsection{Relative bounded combinatorics}
We are ready to define the notion of {\em relative bounded combinatorics} and {\em height} that we will use: We fix, once and for all, a sufficiently small threshold $\delta>0$. 

\begin{dfn}[Relative Bounded Combinatorics]
Consider $Y,X\in\T$. We say that $(Y,X)$ has {\em relative $\delta$-bounded combinatorics} with respect to the handlebody $H_g$ if the Teichm\"uller geodesic $[Y,X]$ is contained in $\T_\delta$ and 
\[
d_\mc{C}(\mc{D},\Upsilon(Y))+d_\mc{C}(\Upsilon(Y),\Upsilon(X))\le d_\mc{C}(\mc{D},\Upsilon(X))+\frac{1}{\delta}.
\]
The {\em height} of the pair $(Y,X)$ is $d_\T(Y,X)$. 
\end{dfn}

Notice that, since $[Y,X]\subset\T_\delta$, by work of Hamenst\"adt \cite{H10}, we have $d_{\mc{C}}(\Upsilon(Y),\Upsilon(X))\ge d_\T(Y,X)/L_\delta-C_\delta$ for some uniform $L_\delta,C_\delta$ provided that $d_\T(Y,X)\ge L_\delta$. In particular, $d_{\mc{C}}(\Upsilon(X),\mc{D})$ is coarsely uniformly bounded from below by the height $d_\T(Y,X)$ provided that this is sufficiently large.

We have the following properties:
 
\begin{lem}
\label{disks disappear}
Fix $g\ge 2$ and $\delta>0$. Let $(Y_n,X_n)$ be a sequence of pairs that have relative $\delta$-bounded combinatorics with respect to $H_g$ and heights $h_n\uparrow\infty$. Then:
\begin{enumerate}
\item{We have
\[
 (\zeta|\Upsilon(Y_n))_{\Upsilon(X_n)}\rightarrow\infty
\]
uniformly in $\zeta\in\mc{D}$.}
\item{If $\phi_n\in{\rm Mod}(\Sigma)$ are mapping classes such that $\phi_nX_n$ lies in a fixed compact set of $\T_\delta$, then, up to passing to subsequences, the sequence of geodesics $\phi_n[X_n,Y_n]$ converges to a Teichm\"uller ray which is entirely contained in $\T_\delta$ and converges to a uniquely ergodic filling lamination $[\lambda]\in\mc{PML}$. Moreover, $\phi_n\mc{D}$ and $\Upsilon(\phi_nY_n)$ both converge in $\mc{C}\cup\partial\mc{C}$ to the point of $\partial\mc{C}$ defined by $\lambda$.}
\item{The length on $X_n$ of the shortest compressible curve in $H_g$ diverges.}
\end{enumerate}
\end{lem}

In the proof and in the sequel we use the following notations:

\begin{center}
\label{notation2}
\begin{minipage}{.8\linewidth}
\textbf{Notation}. If $X$ is a hyperbolic surface and $\gamma:S^1\to X$ is a smooth closed curve, then we denote by $L(\gamma)$ the length of $\gamma$, and by $L_X(\gamma)$ the length of the geodesic representative of $\gamma$ on $X$. For a curve $\gamma$ in a hyperbolic 3-manifold $M$ we use the notation $l(\gamma)$ and $l_M(\gamma)$ for the analogous quantities.
\end{minipage}
\end{center}

\begin{proof}
{\bf Property (1)}. For simplicity, define $\alpha_n:=\Upsilon(Y_n)$ and $\beta_n:=\Upsilon(X_n)$. Notice that $d_{\mc{C}}(\alpha_n,\beta_n)\ge h_n/L_\delta-C_\delta$ because $\Upsilon$ restricted to $[Y_n,X_n]$ is a parametrized $(L_\delta,C_\delta)$-quasi geodesic by \cite{H10}. In particular $d_\mc{C}(\alpha_n,\beta_n)\rightarrow\infty$.

We show that $(\zeta|\alpha_n)_{\beta_n}\gtrapprox d_{\mc{C}}(\alpha_n,\beta_n)$ for every $\zeta\in\mc{D}$ (here the symbol $\gtrapprox$ means greater up to a uniform additive constant). Consider the nearest point projection $\ol{\alpha}_n:=\pi_\mc{D}(\alpha_n)$. Recall that $\mc{C}$ is Gromov hyperbolic. By basic properties of Gromov products in Gromov hyperbolic spaces, we have
\[
(\zeta|\alpha_n)_{\beta_n}\gtrapprox\min\{(\ol{\alpha}_n|\alpha_n)_{\beta_n},(\zeta|\ol{\alpha}_n)_{\beta_n}\}.
\]
So, it is enough to show that both $(\ol{\alpha}_n|\alpha_n)_{\beta_n}$ and $(\zeta|\ol{\alpha}_n)_{\beta_n}$ are at least $d_\mc{C}(\alpha_n,\beta_n)$. Consider first $(\ol{\alpha}_n|\alpha_n)_{\beta_n}$. We have:
\begin{align*}
d_\mc{C}(\beta_n,\ol{\alpha}_n) &\ge d_\mc{C}(\beta_n,\mc{D})\\
 &\ge d_\mc{C}(\alpha_n,\beta_n)+d_\mc{C}(\alpha_n,\mc{D})-1/\delta\\
 &=d_\mc{C}(\alpha_n,\beta_n)+d_\mc{C}(\alpha_n,\ol{\alpha}_n)-1/\delta.
\end{align*}
Thus, we get $(\ol{\alpha}_n|\alpha_n)_{\beta_n}\ge d_\mc{C}(\alpha_n,\beta_n)-1/2\delta$. 

Now consider $(\zeta|\ol{\alpha}_n)_{\beta_n}$. Let $\ol{\beta}_n:=\pi_\mc{D}(\beta_n)$ be the nearest point projection of $\beta_n$ to $\mc{D}$. By quasi-convexity of $\mc{D}$, for every $\gamma\in\mc{D}$ we have $d_\mc{C}(\beta_n,\gamma)\approx d_\mc{C}(\beta_n,\ol{\beta}_n)+d_\mc{C}(\ol{\beta}_n,\gamma)$ (here the symbol $\approx$ means equal up to a uniform additive constant). In particular this holds for $\gamma=\ol{\alpha}_n,\zeta$. Therefore
\begin{align*}
2(\zeta|\ol{\alpha}_n)_{\beta_n} &=d_\mc{C}(\beta_n,\zeta)+d_\mc{C}(\beta_n,\ol{\alpha}_n)-d_\mc{C}(\zeta,\ol{\alpha}_n)\\
 &\approx 2d_\mc{C}(\beta_n,\ol{\beta}_n)+d_\mc{C}(\ol{\beta}_n,\zeta)+d_\mc{C}(\ol{\beta}_n,\ol{\alpha}_n)-d_\mc{C}(\zeta,\ol{\alpha}_n)\\
  &\ge 2d_\mc{C}(\beta_n,\ol{\beta}_n)=2d_\mc{C}(\beta_n,\mc{D})\ge 2d_\mc{C}(\alpha_n,\beta_n)-2/\delta.
\end{align*}

{\bf Property (2)}. Since $\phi_n[X_n,Y_n]\subset\T_\delta$, $d_\T(\phi_nX_n,\phi_nY_n)=h_n\uparrow\infty$, and $\phi_nX_n$ lies in a fixed compact subset of $\T_\delta$, up to extracting subsequences, the sequence of geodesics $\phi_n[X_n,Y_n]$ converges uniformly on compact subsets to a geodesic ray $l:[0,\infty)\rightarrow\T$ that stays in $\T_\delta$. By work of Masur \cite{Mas80}, \cite{Mas92}, the ray $l$ converges to the projective class of a filling uniquely ergodic measured lamination $[\lambda]\in\mc{PML}$. This implies that $\phi_nY_n$ converges to $[\lambda]$ as well. By a result of Klarreich \cite{K99}, $[\lambda]$ defines a point on $\partial\mc{C}$ and the sequence of simple closed curves $\Upsilon(\phi_nY_n)$ converges to it. 

Denote by $\phi_n\alpha_n=\Upsilon(\phi_nY_n),\phi_n\beta_n=\Upsilon(\phi_nX_n)$ a pair of shortest closed geodesics for the metrics $\phi_nY_n,\phi_nX_n$. Since $\phi_nX_n$ is converging, we can assume that $\phi_n\beta_n$ is constant. Consider a sequence of translates of disks $\phi_n\zeta_n$ with $\zeta_n\in\mc{D}$. In order to prove that $\phi_n\zeta_n$ converges to $\lambda$, it is enough to show that $(\phi_m\zeta_m|\phi_n\alpha_n)_{\phi_n\beta_n}\rightarrow\infty$ as $n,m\rightarrow\infty$. 

Fix $M>0$. By the convergence $\phi_n\alpha_n\rightarrow\lambda$, there exists $N>0$ such that for every $n,m\ge N$ we have $(\phi_m\alpha_m|\phi_n\alpha_n)_{\phi_n\beta_n}\ge M$. By property (1) we can also assume that for every $m\ge N$ we have $(\phi_m\zeta_m|\phi_m\alpha_m)_{\phi_m\beta_m}\ge M$. The claim follows from basic properties of Gromov products: Recall that $\phi_n\beta_n$ is constant
\[
(\phi_m\zeta_m|\phi_n\alpha_n)_{\phi_n\beta_n}\gtrapprox\min\{(\phi_m\zeta_m|\phi_m\alpha_m)_{\phi_m\beta_m},(\phi_m\alpha_m|\phi_n\alpha_n)_{\phi_n\beta_n}\}\ge M
\]
for every $n,m\ge N$. 

{\bf Property (3)}. Let $\gamma_n$ be a shortest goedesic for $X_n$ which is compressible in $H_g$. We first show that $\gamma_n$ is simple: Notice that $\gamma_n$ is primitive. We slightly perturb $\gamma_n$ at the self intersections to represent it as a 4-valent graph on $X_n$. By the Loop Theorem there is a simple cycle in this graph that represents a diskbounding curve $\zeta_n$. Such a surgery of $\gamma_n$ has smaller length. Since $\gamma_n$ is the shortest curve which is compressible in $H_g$, we must have that $\gamma_n$ is simple. 

Consider $\beta_n=\Upsilon(X_n)$. Denote by $i(\cdot,\cdot)$ the geometric intersection number between two simple closed curves. Since $L_{X_n}(\beta_n)$ is uniformly bounded away from 0 and $\infty$, a standard consequence of the Collar Lemma gives us a constant $c>0$ such that $L_{X_n}(\gamma_n)\ge c\cdot i(\beta_n,\gamma_n)$. Thus, it is enough to show that $i(\beta_n,\gamma_n)\rightarrow\infty$. This comes from the fact that $\gamma_n\in\mc{D}$ is compressible, the distances $d_\mc{C}(\beta_n,\mc{D})$ diverge, and the fact that distances in the curve graph are bounded by intersections $d_{\mc{C}}(\beta_n,\gamma_n)\le 2i(\beta_n,\gamma_n)+2$ (see Lemma 2.1 in \cite{MM99}).
\end{proof}

\subsection{The proof of Proposition \ref{large-thick collar}}
We are now ready to prove the main result of this section. 

The strategy is easy to state: We argue by contradiction. Suppose that, for some fixed $\delta >0$ and numbers $L>0,\xi>0$, 
we have a sequence of counterexamples $N_n=H(X_n)$ and $Q_n=Q(Y_n,X_n)$ with relative $\delta$-bounded combinatorics and diverging heights, but not satisfying the conclusion of the proposition.

Using the results from Section \ref{hyperbolicstructures}, we can pass to geometric limits $f:X\rightarrow N$ and $f':X'\rightarrow Q$ of the sequences $\partial\mc{CC}(N_n)\subset N_n$ and $\partial_{X_n}\mc{CC}(Q_n)\subset Q_n$ (the component of the boundary of the convex core facing the conformal boundary $X_n$). The main point of the proof is the following claim:

{\bf Claim}: The hyperbolic manifolds $N$ and $Q$ are {\em singly degenerate} hyperbolic structures on $\Sigma\times\mb{R}$ with the same {\em end invariants}.

For a comprehensive discussion of ends and end invariants of hyperbolic 3-mani\-folds, we refer to Section 2 of \cite{M10}.

Once we know that the claim holds, the solution of the Ending Lamination Conjecture by Minsky \cite{M10} and Brock-Canary-Minsky \cite{BCM12} will tell us that $Q$ and $N$ are isometric via an orientaion preserving isometry in the correct homotopy class. By the definition and properties of geometric convergence, this is enough to find $\xi$-almost isometric embeddings in the right homotopy classes of collars of any arbitrary fized size $k:V_n\subset\mc{CC}(Q_n)\rightarrow U_n\subset\mc{CC}(N_n)$ for all sufficiently large $n$. Thus we will obtain a contradiction to the initial assumptions.

For convenience, we divide the proof of the claim into several small steps. 

We first consider the limit geometric limit $f:X\rightarrow N$. To begin with, recall that in our setup we have a diagram, commutative up to homotopy, which is provided by geometric convergence
\[
 \xymatrix{
 \partial\mc{CC}(N_n)\ar[r]^(0.60){f_n} &N_n\\
 X\ar[r]_{f}\ar[u]^{\phi_n} &N\ar[u]_{k_n}.\\
 }
\]
Here the vertical arrows are the approximating maps, $f_n:\partial\mc{CC}(N_n)\rightarrow N_n$ is the inclusion of the boundary of the convex core, and $f:X\rightarrow N$ is the limit pleated surface in $N$. 

We give an arbitrary marking to $X$ in order to identify it to a point in $\T$. For simplicity, with a little abuse of notations, we will also denote by $\phi_n$ the homotopy class of the map $\phi_n:\partial\mc{CC}(N_n)\rightarrow X$ with respect to the markings on $\partial\mc{CC}(N_n)$ and $X$. With this notation we have that $\phi_n^{-1}\partial\mc{CC}(N_n)$ converges to $X$ as a sequence of points in $\T$. 

\begin{lem}
\label{limit incompressible}
The map $f$ is incompressible.
\end{lem}

\begin{proof}
Suppose that $f$ is compressible. Let $\gamma\subset X$ be an essential closed curve such that $f(\gamma)$ is null-homotopic in $N$. If $n$ is large enough, then we can transport a null-homotopy of $f(\gamma)$ to $N_n$ using the approximating map $k_n$. Since $k_nf$ is locally homotopic to $f_n\phi_n$, we conclude that $\phi_n(\gamma)\subset\partial\mc{CC}(N_n)$ is null-homotopic in $N_n$ for every large $n$. But the length of $\phi_n(\gamma)$ on $\partial\mc{CC}(N_n)$ is roughly the length of $\gamma$ on $X$, in particular it is bounded. By Theorem \ref{bridgeman-canary}, $\phi_n(\gamma)$ has also uniformly bounded length on $X_n$. This contradicts property (3) of Lemma \ref{disks disappear}.
\end{proof}

Let us consider now the $f_*\pi_1(X)$-covering of $N$ which we denote by $p:\ol{N}\rightarrow N$. By covering theory, the map $f:X\rightarrow N$ lifts to $\ol{N}$, and the lift $\ol{f}:X\rightarrow\ol{N}$ is a homotopy equivalence. We fix once and for all such a lift $\ol{f}$. 

By work of Thurston \cite{Th79} and Bonahon \cite{Bo86}, we know that, in this setting, $\ol{N}$ is homeomorphic to $\Sigma\times\mb{R}$.

The next step of the proof consists in analizing the ends of $\ol{N}$: We show that $\ol{N}$ has a visible {\em geometrically finite} end and another one which is {\em simply degenerate}, both homeomorphic to $\Sigma\times[0,\infty)$.

\begin{lem}
\label{divergence of disks}
Up to passing to a subsequence of the sequence $N_n$, we have
\begin{itemize}
\item{The sequence $\phi_n^{-1}X_n$ converges to $X_\infty\in\T_\delta$.}
\item{There exists a filling lamination $\lambda\in\partial\mc{C}$ with the following property: If $\zeta_n\in\mc{D}$ is a simple closed diskbounding curve for each $n$, then the sequence of simple closed curves $\phi_n^{-1}\zeta_n$ converges to $\lambda$. Furthermore, $\lambda$ is also the limit of the sequence $\phi_n^{-1}\Upsilon(Y_n)$.}
\end{itemize}
\end{lem}

\begin{proof}
The second point follows from the first one and property (2) of Lemma \ref{disks disappear}. We only have to show that the maps $\phi_n^{-1}$ are such that $\phi_n^{-1}X_n$ lies in a fixed compact subset of $\T_\delta$. This is a consequence of the fact that $X_n$ and $\phi_n^{-1}\partial\mc{CC}(N_n)$ have coarsely the same length spectrum by Theorem \ref{bridgeman-canary} and the fact that $\phi_n^{-1}\partial\mc{CC}(N_n)$ converges to $X$ in $\T$. 
\end{proof}

The hyperbolic structure $X_\infty$ will be the conformal boundary of the geometrically finite end of $\ol{N}$ while $\lambda$ will be the ending lamination of the simply degenerate end.

\begin{lem}
\label{convex cocompact end}
$\ol{N}$ has a geometrically finite end $\ol{E}$ homeomorphic to $\Sigma\times[0,\infty)$. 
\end{lem}

\begin{proof}
We show that there is a proper closed convex subset $C\subset\ol{N}$ that contains the convex core $\mc{CC}(\ol{N})$. This immediately implies that, up to removing standard neighbourhoods of cusps, $\partial C$ bounds a neighbourhood of a geometrically finite end of $\ol{N}$ homeomorphic to $\partial C\times(0,\infty)$. By the basic structure of the ends of the hyperbolic 3-manifold $Q\simeq\Sigma\times\mb{R}$, if $\partial C$ is a closed surface, then $\partial C\simeq\Sigma$.

We now produce the convex set $C\subset\ol{N}$ as a limit of the convex cores $\mc{CC}(N_n)$. In order to do so, it is convenient to work on the universal coverings: We identify $(N_n,x_n),(N,x)$ with $(\mb{H}^3/\Gamma_n,o_3),(\mb{H}^3/\Gamma,o_3)$, where $o_3\in\mb{H}^3$ is a fixed basepoint. Consider the convex hulls $\mc{CH}_n=\mc{CH}(\Lambda_n),\mc{CH}=\mc{CH}(\Lambda)$ of the limit sets $\Lambda_n,\Lambda$ of $\Gamma_n,\Gamma$. We have that, up to subsequences, $\mc{CH}_n$ converges in the Chabauty topology on closed subsets of $\mb{H}^3$ (see Chapter E.1 of \cite{BP92}) to a closed set $K\subset\mb{H}^3$. It follows from general properties of the Chabauty topology (see Proposition E.1.2 of \cite{BP92}) that $K$ is convex and invariant under $\Gamma$. Moreover, as $o_3\in\partial\mc{CH}_n$ for all $n$, $K$ is also proper. Invariance under $\Gamma$ implies that $\mc{CH}\subset K$ so that the quotient $C:=K/f_*\pi_1(X)\subset\ol{N}$ is a convex set containing the convex core of $\ol{N}$.

We now show that the boundary $\partial C=\partial K/f_*\pi_1(X)$ is a closed surface just by observing that it is the image $\ol{f}(X)=\partial C$ of the closed surface $X$ under the map $\ol{f}$. Again, it is convenient to work on the universal coverings: We identify $(\partial\mc{CC}(N_n),x_n)$ and $(X,x)$ with $(\mb{H}^2/\Theta_n,o_2),(\mb{H}^2/\Theta,o_2)$ where $o_2\in\mb{H}^2$ is a fixed basepoint, and coherently lift $f_n,f$ to equivariant pleated maps $F_n,F:(\mb{H}^2,o_2)\rightarrow(\mb{H}^3,o_3)$ with $F_n(\mb{H}^2)=\partial\mc{CH}_n$. Notice that the boundary $\partial K$ is the Chabauty limit of the boundaries $\partial\mc{CH}_n$. Since $F_n\rightarrow F$ uniformly on compact sets, we have $F(\mb{H}^2)=\partial K$. In particular, $\ol{f}(X)=\partial C$. 
\end{proof}

It also follows from the proof that the restriction of the covering projection $p$ to $\ol{E}$ is a homeomorphism onto a neighbourhood of a geometrically finite end $E$ of $N$ bounded by $f(X)$. Namely, $f(X)=\partial K/\Gamma\subset N$ also is 
an embedded surface bounding the convex set $K/\Gamma\subset N$. The covering projection $p$ restricts to a covering projection $p:\partial C\rightarrow f(X)$. 
This covering is a homeomorphism provided that $p_*\pi_1(\partial C)=\pi_1(f(X))$, and in this case
the restriction of $p$ to $\ol{E}$ is a homeomorphism onto the image which is a neighbourhood of a geometrically finite end homeomorphic to $\Sigma\times[0,\infty)$. 

By construction, we have $p_*\pi_1(\partial C)=f_*\pi_1(X)$.
A priori, $f_*\pi_1(X)$ might be different from $\pi_1(f(X))$. But 
$f$ is locally homotopic an embedding (such as $k_n^{-1}f_n\phi_n$ for $n$ large enough) 
and can be perturbed to an embedding, with image contained in a small tubular neighbourhood of $f(X)$ of the form $f(X)\times[-1,1]$. 
Since $f$ is also $\pi_1$-injective, by standard 3-manifold topology (see Proposition 3.1 of \cite{W68}), we conclude that $f_*\pi_1(X)=\pi_1(f(X))$.

We now prove that $\ol{N}$ has also a simply degenerate end. Using the fact that $\lambda$ is a limit of translates of disk bounding curves, we show the following:

\begin{lem}
\label{lamination not realized}
The lamination $\lambda$ is not realized in $\ol{N}$.
\end{lem}

\begin{proof}
Recall that, by Lemma \ref{divergence of disks}, $\lambda$ is a limit in $\mc{C}\cup\partial\mc{C}$ of the sequence $\phi_n^{-1}\zeta$ where $\zeta\in\mc{D}$ is a fixed diskbounding curve. Once we fixed an auxiliary hyperbolic structure on $\Sigma$, this implies that $\lambda$ is contained in a Hausdorff limit of the sequence of geodesic realizations of the simple closed curves $\phi_n^{-1}\zeta$. Suppose we can realize $\lambda$ in the homotopy class of the map $\ol{f}:X\rightarrow\ol{N}$. Then, a standard train-track approximation argument (see the discussion above Theorem I.5.3.10 in \cite{CEG06}), tells us that for sufficiently large $n$ the curve $\phi_n^{-1}\zeta$ is also realized as a closed geodesic, in a bounded neighbourhood of a realization of $\lambda$ in $\ol{N}$. By composition with the covering projection, the curve $\phi_n^{-1}\zeta$ can also be realized in $N$ in the homotopy class of $f:X\rightarrow N$. By geometric convergence, this implies that for large $n$, we can represent the curve $\zeta$ in $N_n$ as a curve with very small geodesic curvature. Such a curve is not nullhomotopic. But this is absurd as $\zeta$ is compressible in $N_n$.  
\end{proof}

By Proposition 9.7.1 of \cite{Th79} or Theorem 1.4 of \cite{NS12}, we deduce: 

\begin{cor}
\label{simply degenerate end}
$\ol{N}$ has a simply degenerate end homeomorphic to $\Sigma\times[0,\infty)$ with ending lamination $\lambda$. 
\end{cor}

To conclude, we found that $\ol{N}$ is a hyperbolic structure on $\Sigma\times\mb{R}$ for which one of the end invariants is a filling lamination, and the second is a marked conformal structure on $\Sigma$. Since there is no room for other ends, we see that $\ol{N}$ is singly degenerate. This immediately implies:

\begin{lem}
\label{trivial covering}
The covering $p:\ol{N}\rightarrow N$ is trivial.
\end{lem}

\begin{proof} 
If the covering $p:\ol{N}\rightarrow N$ is {\em not} trivial, then, as the restriction of $p$ to the geometrically finite end $\ol{E}$ of $\ol{N}$ is a homeomorphism onto a geometrically finite end $E$ of $N$, there exists at least one other preimage of $E$. This preimage then is a geometrically finite end of $\ol{N}$  different from $\ol{E}$. But $\ol{N}$ has a single geometrically finite end and hence we get a contradiction.
\end{proof}

We now identify $N=\ol{N}$ and compute the conformal boundary of the geometrically finite end:

\begin{lem}
\label{conf boundary}
The end $E$ is conformally compactified by $X_\infty$.
\end{lem}

\begin{proof}
We isometrically identify $(N_n,x_n)$ and $(N,x)$ with $(\mb{H}^3/\Gamma_n,o_3)$ and $(\mb{H}^3/\Gamma,o_3)$ where $o_3\in\mb{H}^3$ is a fixed origin. Let $\pi_1(N_n,x_n)\rightarrow\Gamma_n<{\rm Isom}^+(\mb{H}^3)$ and $\pi_1(N,x)\rightarrow\Gamma<{\rm Isom}^+(\mb{H}^3)$ be the homonomy identifications. Let $\Lambda_n$ and $\Lambda$ be the limit sets of $\Gamma_n$ and $\Gamma$ and let $\Omega_n$ and $\Omega$ be the domains of discontinuity (notice that all of them are connected). 

The main observation is that $\Lambda_n\rightarrow\Lambda$ in the Hausdorff topology of $\partial\mb{H}^3$: This follows from Lemma 7.33 of \cite{MT98} (see also Kerckhoff-Thurston \cite{KT90}) combined with a result of Bowditch \cite{Bow13} that gives a uniform upper bound on ${\rm inj}_x(N_n)$ for every $x\in\mc{CC}(N_n)$. We now use this fact to compute directly the uniformization of $\Omega$.

Fix $z\in\Omega$ a basepoint. Since $\Lambda_n\rightarrow\Lambda$, we have that $z\in\Omega_n$ for every $n$ large enough. Denote by $\pi_n:(\mb{H}^2,o_2)\rightarrow(\Omega_n/\Gamma_n,z)$ the universal covering projection where $o_2\in\mb{H}^2$ is a fixed origin. Let $\pi_1(\Omega_n/\Gamma_n,z)\rightarrow{\rm Isom}^+(\mb{H}^2)$ be the corresponding holonomy representation. Notice that $\pi_n$ factors through the covering $\beta_n:(\mb{H}^2,o_2)\rightarrow(\Omega_n,z)$ corresponding to the subgroup ${\rm ker}\{\pi_1(\Omega_n/\Gamma_n)\rightarrow\pi_1({\hat N_n})\}$.

Consider the representations
\[
\sigma_n:\pi_1(X)\stackrel{\phi_n}{\longrightarrow}\pi_1(\partial\mc{CC}(N_n))\simeq\pi_1(\Omega_n/\Gamma_n)\rightarrow{\rm Isom}^+(\mb{H}^2).
\]
Since the conformal boundaries $\phi_n^{-1}X_n$ converge to $X_\infty$, we have that $\sigma_n$ converges to a representation $\sigma_\infty$ such that $\mb{H}^2/\sigma_\infty$ is isometric to $X_\infty$. Similarly, consider the representations 
\[
\rho_n:\pi_1(X)\stackrel{\phi_n}{\longrightarrow}\pi_1(\partial\mc{CC}(N_n))\rightarrow\pi_1(N_n)\rightarrow{\rm Isom}^+(\mb{H}^3).
\]
Again, the sequence of representations converges to $\rho_\infty$ which is an isomorphism between $\pi_1(X)$ and $\Gamma$. The covering maps $\beta_n$ are $\pi_1(X)$-equivariant with respect to $\sigma_n$ and $\rho_n$.

We show that $\beta_n$ converges uniformly on compact sets to a biholomorphism $\beta:\mb{H}^2\rightarrow\Omega$ that is equivariant with respect to $\sigma_\infty$ and $\rho_\infty$. Equivariance will be automatic once we prove that $\beta_n$ converges uniformly on compact sets.

Since $\Lambda_n\rightarrow\Lambda$, up to slightly changing the basepoints $x_n$, we can assume that for $n$ large enough $\Lambda_n$ passes through a fixed triple of points $\{p_1,p_2,p_3\}\subset\Lambda$ so that $\Omega_n\subset\partial\mb{H}^3-\{p_1,p_2,p_3\}$. Therefore, by Montel's and Hurwitz's Theorems (see for example \cite{Mil06}), we have that the sequence of locally univalent holomorphic maps $\beta_n:\mb{H}^2\rightarrow\Omega_n$ converges uniformly on compact sets to a locally univalent holomorphic map $\beta:\mb{H}^2\rightarrow\partial\mb{H}^3$. Furthermore, by one half of Caratheodory's Kernel Theorem (see Theorem 7.30 in \cite{MT98}), we have $\beta(\mb{H}^2)=\Omega$. 

The equivariant locally univalent holomorphic map $\beta$ descends to a holomorphic covering map $\mb{H}^2/\sigma_\infty\rightarrow\Omega/\Gamma$. Such a map must be a biholomorphism as the source and the target are homeomorphic.
\end{proof}

By the discussion so far, we have identified $N$ with a {\em singly degenerate} structure on $\Sigma\times\mb{R}$, and we computed the end invariants. We now repeat the same analysis for the limit $f':X'\rightarrow Q$ of the sequence $f_n':\partial_{X_n}\mc{CC}(Q_n)\rightarrow Q_n$ of quasi-fuchsian manifolds and its $f'_*\pi_1(X')$-covering $\ol{Q}$. The same exact arguments given above work also in this case except for the proof of non-realizability of $\lambda$:

\begin{proof}[Proof of Corollary \ref{simply degenerate end} for $\ol{Q}$]
There are several simple ways to proceed here. For the sake of brevity, and since we will use this result also later on, we just invoke Theorem 1.1 of \cite{BBCM13}: It is enough to observe that the sequence of quasi-fuchsian manifolds $\phi_n^{-1}Q_n$ converges {\em algebraically} to $\ol{Q}$ (because the maps $f_n\phi_n$ converge to $f$) and that $\Upsilon(\phi_n^{-1}Y_n)\rightarrow\lambda$ by Lemma \ref{divergence of disks}.
\end{proof}

In conclusion, $Q$ and $N$ are hyperbolic structures on $\Sigma\times\mb{R}$ with the same end invariants. Thus, the solution of the Ending Lamination Conjecture \cite{M10}, \cite{BCM12} provides us an orientation preserving isometry $Q\rightarrow N$ in the correct homotopy class. Using this isometry and the approximating maps from the geometric convergences $Q_n\rightarrow Q$ and $N_n\rightarrow N$ we obtain the desired $\xi$-almost isometric embeddings of a large collar of the component of $\partial\mc{CC}(Q_n)$ facing $X_n$ into a collar of $\partial\mc{CC}(N_n)$. This contradicts the initial assumptions and finishes the proof of Proposition \ref{large-thick collar}. \qed

\section{Gluing}
\label{almostisometric}

In this section we prove the gluing theorem.

We briefly recall our setup: We consider four points $Y<X<X'<Y'$ aligned on the Teichm\"uller segment $[Y,Y']$ and associate to them the hyperbolic structures an $N_1=H(X)$, $Q=Q(Y,Y')$, $N_2=H(f^{-1}X')$. We want to cut from their convex cores gluing blocks $N_0^1\subset\mc{CC}(N_1)$, $Q_0\subset\mc{CC}(Q)$, $N_0^2\subset\mc{CC}(N_2)$, produce identifications $k_j:V_j\subset Q_0\rightarrow U_j\subset N_0^j$ between their collars, and then apply the cut and glue construction Lemma \ref{cut-and-glue}.

In Proposition \ref{unif emb} and Corollary \ref{cor unif emb}, we show that, under suitable assumptions on 
the pairs $(Y,X)$ and $(Y',X')$, there are product regions in $\mc{CC}(Q(Y,X)),\mc{CC}(Q(Y',X'))$ which {\em simultaneously} $\xi$-almost isometrically embed 
into $\mc{CC}(H(X)),\mc{CC}(H(f^{-1}X))$ and $\mc{CC}(Q(Y,Y'))$. Then we proceed and show in Lemma \ref{position} and Lemma \ref{middle gluing block} that, under suitable assumptions, the embeddings of such product regions in $\mc{CC}(Q(Y,Y'))$ can be chosen to have disjoint images so that they cobound a gluing block $Q_0$. This is the last ingredient needed for applying Lemma \ref{cut-and-glue} and prove Theorem \ref{gluing}. 

\subsection{Quasi-fuchsian manifolds with bounded geometry}
In the proof of the main results of the section, we will use the following fundamental relation between the geometry of Teichm\"uller space and the geometry of quasi-fuchsian manifolds:  

\begin{thm}[Rafi \cite{R05}, see also Minsky \cite{M01}]
\label{bounded}
Fix $g\ge 2$. For every $\delta>0$ there exists $\ep>0$ such that if $[Y,X]\subset\T_\delta$, then ${\rm inj}(Q(Y,X))\ge\ep$. 
\end{thm}

It follows from work of Thurston that the space of pointed hyperbolic manifolds $Q$ diffeomorphic to $\Sigma\times\mb{R}$ and with ${\rm inj}(Q)\ge\ep$ is compact in the geometric topology (see Theorem 4.3 and Corollary 4.4 in \cite{M93}). In particular, a simple compactness argument implies the following property which will be useful for us later on: 

\begin{lem}
\label{large implies product}
Fix $\ep>0$. For every size $D>0$ there exists $D'>0$ such that for every quasi-fuchsian manifold $Q=Q(Y,X)$ with ${\rm inj}(Q)\ge\ep$ and every point $x\in\mc{CC}(Q(Y,X))$ with $d_Q(x,\partial\mc{CC}(Q))\ge D'$, there is a product region $U\subset\mc{CC}(Q)$ of size $D$ with $x\in U$.
\end{lem}

\subsection{Embedding product regions in quasi-fuchsian manifolds}
Our next goal is to find simultaneous embeddings of product regions of $Q(Y,X)$ in $\mc{CC}(H(X))$ and $\mc{CC}(Q(Y,Y'))$. 

\begin{pro} 
\label{unif emb}
For every $\delta,\xi,D$ there exists $L>0$ such that the following holds: Suppose that we have a product region $U\subset\mc{CC}(Q)$ in a quasi fuchsian manifold $Q=Q(Y,X)$. Let $Q'=Q(Z,Z')$ be another quasi-fuchsian manifold with $[Y,X]\subset[Z,Z']$. Suppose that $U$ has size $D$, $d_Q(U,\partial\mc{CC}(Q))\ge L$ and $[Y,X]\subset\T_\delta$. Then, there exists an orientation preserving $\xi$-almost isometric embedding $k:U\rightarrow\mc{CC}(Q')$ in the homotopy class of the identity.
\end{pro}

Combining Proposition \ref{large-thick collar}, Proposition \ref{unif emb} (with $[Z,Z']=[Y,Y']$), and Lemma \ref{large implies product}, we get:

\begin{cor}
\label{cor unif emb}
Fix $\delta,\xi,D>0$. Then there exists $h=h(\delta,\xi,D)>0$ such that the following holds: Suppose that $(Y,X)$ has relative $\delta$-bounded combinatorics and height at least $h$. Consider the convex cocompact handlebody $H(X)$ and a quasi-fuchsian manifold $Q(Y,Y')$ for which $[Y,X]\subset[Y,Y']$. Then there exist a product region $U\subset\mc{CC}(H(X))$ of size $D$ in a collar of the boundary of the convex core, a product region $V\subset\mc{CC}(Q(Y,Y'))$ of size $D$ and a $\xi$-almost isometric orientation preserving diffeomorphism $k:V\rightarrow U$ in the homotopy class of the identity.
\end{cor}

\begin{proof}[Proof of Corollary \ref{cor unif emb}]
Let $L>0$ be a large arbitrary size. By Proposition \ref{large-thick collar}, if $h$ is large enough, we can find a $\xi$-almost isometric diffeomorphism between a collar of width $2L+4D$ of $\partial_X\mc{CC}(Q=Q(Y,X))$ and a collar of $\partial\mc{CC}(H(X))$. By Lemma \ref{large implies product}, we can find a product region $U$ of size $D$ in the $Q$-collar with $d_Q(U,\partial\mc{CC}(Q))\ge L$. By Proposition \ref{unif emb}, if $L$ is sufficiently large, such product $\xi$-almost isometrically embeds also in $\mc{CC}(Q(Y,Y'))$. All the embeddings are in the homotopy class of the identity.
\end{proof}

We now prove the proposition.

\begin{proof}[Proof of Proposition \ref{unif emb}]
We argue by contradiction. Consider a sequence of counterexamples $U_n\subset Q_n=Q(Y_n,X_n)$ and $Q_n'=Q(Z_n,Z_n')$ with $d_{Q_n}(U_n,\partial\mc{CC}(Q_n))\ge n$ and such that there is no $\xi$-almost isometric embedding of $U_n$ in $\mc{CC}(Q_n')$. 

We show that, after suitable choices of basepoints, the sequences $Q_n$ and $Q_n'$ converge geometrically to the same doubly degenerate structure on $\Sigma\times\mb{R}$ so that we can almost isometrically embed $U_n$ in $Q_n'$ for large enough $n$. This will provide the desired contradiction. 

As a first step we consider the sequence $Q_n$. 

By the Canary-Thurston Filling Theorem (see Canary \cite{C96} and Theorem 9.5.13 of Thurston \cite{Th79}), in a uniform neighbourhood of each point in $\mc{CC}(Q_n)$ there exists a pleated surface. Thus, we can pick a pleated surface $f_n:W_n\rightarrow Q_n$ that passes through a point $x_n\in Q_n$ unifromly close to $U_n\subset\mc{CC}(Q_n)$. 

Notice that, by Theorem \ref{bounded}, the injectivity radius of $Q_n$ is uniformly bounded from below by $\text{\rm inj}(Q_n)\ge\ep$ where $\ep$ only depends on $\delta$ and on $g$. In particular, as $\text{\rm inj}(W_n)\ge\text{\rm inj}(Q_n)$, we have that $W_n\in\T_\ep$. Since the mapping class group acts cocompactly on $\T_\ep$, up to coherently remarking $Q_n$ and $W_n$, we can assume that $W_n$ lies in a fixed compact subset of $\T_\ep$. 

We now take a geometric limit of the sequence $f_n:W_n\rightarrow Q_n$. The limit is $f:W\rightarrow Q_\infty$ where $f:W\rightarrow Q_\infty$ is a pleated surface. As $\text{\rm inj}(Q_n)\ge\ep$, we also have $\text{\rm inj}(Q_\infty)\ge\ep$. 

\begin{lem}
$Q_\infty$ is a doubly degenerate structure on $\Sigma\times\mb{R}$.
\end{lem}

\begin{proof}
In the absence of parabolics in the limit (recall that $\text{\rm inj}(Q_\infty)\ge\ep$), a result of Thurston (see Theorem 9.2 in \cite{Th79} and also Theorem 4.3 in \cite{M93}), implies that the geometric limit $Q_\infty$ and its $f_*\pi_1(W)$-covering coincide. In particular, $Q_\infty$ is a hyperbolic structure on $\Sigma\times\mb{R}$. Since we assumed that $d_{Q_n}(U_n,\partial\mc{CC}(Q_n))$ diverges, each end of $Q_\infty$ must be simply degenerate (for example because every point in $Q_\infty$ lies uniformly close to a closed geodesic). Thus $Q_\infty$ is a doubly degenerate structure on $\Sigma\times\mb{R}$. 
\end{proof}

Next, we compute the ending laminations of $Q_\infty$. 

Even if, in this case, the only ingredients required are standard arguments mainly due to Thurston (as can be found in Chapter 9 of \cite{Th79} and in \cite{Th2}), for the sake of brevity we exploit instead a simple criterion due to Brock-Bromberg-Canary-Minsky (see Theorem 1.1 of \cite{BBCM13}): If the sequences $\Upsilon(X_n)$ and $\Upsilon(Y_n)$ converge to laminations $\lambda^X$ and $\lambda^Y$ in $\partial\mc{C}$, then $\lambda^X$ and $\lambda^Y$ are respectively the ending laminations of the positive and negative ends of $Q_\infty$.   

We have the following:

\begin{lem}
Up to subsequences, both $X_n$ and $Y_n$ converge in $\T\cup\mc{PML}$ to the projective classes $[\lambda^X]$ and $[\lambda^Y]$ of measured laminations which are uniquely ergodic, filling and together bind the surface. Furthermore, $Z_n\rightarrow[\lambda^Y]$ and $Z_n'\rightarrow[\lambda^X]$.
\end{lem}

\begin{proof}
We start by analizing the relative position of $W_n$ and $X_n$: We know that 
\[
d_{Q_n}(f_n(W_n),\partial_{X_n}\mc{CC}(Q_n))\rightarrow\infty.
\]
This implies (see for example Lemma 4.8 of \cite{M94}) that $d_\T(W_n,\partial_{X_n}\mc{CC}(Q_n))$ diverges as well. Since $d_\T(X_n,\partial_{X_n}\mc{CC}(Q_n))$ is uniformly bounded, we conclude that $d_\T(W_n,X_n)\rightarrow\infty$. Similarly, $d_\T(W_n,Y_n)\rightarrow\infty$.

By Theorem A of \cite{M93}, the hyperbolic structure $W_n$ lies uniformly close to the segment $[Y_n,X_n]$. Since $W_n$ converges to $W$, and the endpoints of the segment are escaping towards $\mc{PML}$, we can assume that, up to subsequences, $[Y_n,X_n]\subset\T_\delta$ converges to a bi-infinite Teichm\"uller line $l\subset\T_\delta$. By results of Masur \cite{Mas80}, \cite{Mas92}, $l$ converges in the forward and backward directions to the projective classes of filling uniquely ergodic laminations $\lambda^X$ and $\lambda^Y$ that together bind the surface $\Sigma$. 

By the properties of the Thurston compactification of Teichm\"uller space, this implies that $X_n\rightarrow[\lambda^X]$ and $Y_n\rightarrow[\lambda^Y]$.
But $[Y_n,X_n]\subset[Z_n,Z_n']$ and consequently $Z_n,Z_n'$ also converge to $[\lambda^Y],[\lambda^X]$.
\end{proof}

Using a result of Klarreich \cite{K99} we get:

\begin{cor}
\label{end inv}
$[\lambda^X],[\lambda^Y]$ determine points $\lambda^X,\lambda^Y\in\partial\mc{C}$, and we have $\Upsilon(X_n),\Upsilon(Z'_n)\rightarrow\lambda^X$ and $\Upsilon(Y_n),\Upsilon(Z_n)\rightarrow\lambda^Y$.
\end{cor}

This finishes the computation of the end invariants of $Q_\infty$.

We now show that, up to subsequences, also $Q_n'$ converges to $Q_\infty$. 

Since $Z_n',Z_n\rightarrow[\lambda^X],[\lambda^Y]$ and $\lambda^X,\lambda^Y$ bind the surface, we can apply Thurston's Double Limit Theorem \cite{Th2} and show that the sequence of quasi-fuchsian manifolds $Q_n'$ has a subsequence that converges algebraically to a manifold $Q_\infty'$. Theorem 1.1 of \cite{BBCM13} and Corollary \ref{end inv} imply that $Q_\infty'$ is a doubly degenerate hyperbolic structure where the ending laminations of the positive and negative ends are respectively $\lambda^X$ and $\lambda^Y$. Since the limit is doubly degenerate, in particular has no accidental parabolics, the convergence $Q_n'\rightarrow Q_\infty'$ is not only algebraic, but also geometric (by Theorem 9.2 of \cite{Th79}).

By the solution of the Ending Lamination Conjecture \cite{M10}, \cite{BCM12}, there is an orientation preserving isometry between $Q_\infty$ and $Q_\infty'$ in the homotopy class of the identity as desired. Since the convergence $Q_n'\rightarrow Q_\infty'$ is both algebraic and geometric, we obtain a $\xi$-almost isometric embedding $U_n\subset\mc{CC}(Q_n)\rightarrow\mc{CC}(Q_n')$ in the homotopy class of the identity for all sufficiently large $n$. This contradicts our initial assumptions and concludes the proof of Proposition \ref{unif emb}.  
\end{proof}

\subsection{Geometry of the middle gluing block}
Let $f\in{\rm Mod}(\Sigma)$ be our gluing map. If $(Y,X)$ and $(f^{-1}Y',f^{-1}X')$ have relative $\delta$-bounded combinatorics with respect to $H_g$ and sufficiently large height, then Corollary \ref{cor unif emb} provides us orientation preserving $\xi$-almost isometric diffeomorphisms $k_1:V_1\subset\mc{CC}(Q(Y,Y'))\rightarrow U_1\subset\mc{CC}(H(X))$ and $k_2:V_2\subset\mc{CC}(Q(Y,Y'))\rightarrow U_2\subset\mc{CC}(H(f^{-1}X'))$ of product regions of size $D$ and where $k_1$ is in the homotopy class of the identity and $k_2$ is in the homotopy class of $f$. 

We now analyze the relative position of $V_1$ and $V_2$ in $\mc{CC}(Q(Y,Y'))$ and show that, under suitable assumptions, $V_1,V_2\subset\mc{CC}(Q(Y,Y'))$ also determine a gluing block $Q_0$ of which they are respectively the bottom and top collars. Since $\mc{CC}(Q(Y,Y'))\simeq\Sigma\times[0,1]$ and both $V_1,V_2$ are parallel to the boundary components, it is enough to prove that $V_1,V_2$ are disjoint and $V_1$ is closer to $\partial_Y\mc{CC}(Q(Y,Y'))$ than $V_2$. 

In order to locate $V_1,V_2$ inside $\mc{CC}(Q(Y,Y'))$ we exploit the geometry of the pleated surfaces with image in $V_1,V_2$. This is the content of the next few lemmas. We begin with the following:

\begin{lem}
\label{position}
For every $\delta,D_0>0$ and there exist $A(\delta)>0$ and $D_1=D_1(D_0,\delta)>0$ such that for every product region $U\subset\mc{CC}(Q)$ of size $D\ge D_1$ in a quasi-fuchsian manifold $Q=Q(Y,X)$ with $[Y,X]\subset\T_\delta$, there exists a hyperbolic surface $Z\in\T$ contained in the $A$-neighbourhood of $[Y,X]$ and a 1-Lipschitz map $f:Z\rightarrow U$ with $d_Q(f(Z),\partial U)\ge D_0$.
\end{lem}

\begin{proof}
Recall that pleated surfaces in $Q$ have uniformly bounded diameter as their hyperbolic structures all live in $\T_\ep$ where $\ep=\ep(\delta)>0$ is the uniform lower bound on the injectivity radius $\text{\rm inj}(Q)\ge \ep$ (Theorem \ref{bounded}). Also recall that for each point in the convex core there is a pleated surface that passes uniformly nearby. Let $D$ be very large. Pick a point $x\in U$ such that $d_Q(x,\partial U)=D/2$ and let $f:Z\rightarrow Q$ be any pleated surface that passes uniformly close to $x$. By Theorem A of Minsky \cite{M93}, we have that $Z$ lies in the $A$-neighbourhood of $[X,Y]$ for some uniform $A=A(\delta)>0$. Since $f(Z)$ has bounded diameter, its distance from $\partial U$ is roughly $D/2$. In particular, if $D$ is large enough, it is larger than $D_0$. 
\end{proof}

If we have a $\xi$-almost isometric diffeomorphism between product regions $k:U\rightarrow V$, where $U$ is as in the previous lemma and $V$ is arbitrary, we can transport the 1-Lipschitz map $f:Z\rightarrow U$ to a 2-Lipschitz map $kf:Z\rightarrow V$ with roughly the same geometric behaviour.

To control better the position of a 2-Lipschitz map $kf:Z\rightarrow V$, we anchor it to a moderate length closed geodesic: 

\begin{lem}
\label{not short}
For every $\ep>0$ there exists $D_0>0$ such that the following holds: Let $V\subset\mc{CC}(Q)$ be a product region with injectivity radius ${\rm inj}(V)\ge\ep$ and containing the image of a 2-Lipschitz map $f:Z\rightarrow V$ satisfying $d_Q(f(Z),\partial V)\ge D_0$. Let $\beta$ be a shortest closed geodesic for $Z$. Then the geodesic representative of $f(\beta)$ in $Q$ lies in $V$ and has length at least $2\ep$. 
\end{lem}

\begin{proof}
Notice that $L_Z(\beta)\le B$ for some uniform $B>0$. The curve $f(\beta)$ has also uniformly bounded length as $l(f(\beta))\le 2L_Z(\beta)\le 2B$. Let $\mb{T}$ be the $\ep$-Margulis tube of the geodesic representative of $f(\beta)$ (it reduces to the core geodesic if $l_Q(f(\beta))\ge 2\ep$). By standard hyperbolic geometry we have $d_Q(f(\beta),\mb{T})\le \log(2l(f(\beta))/\ep)$. Therefore, as $f(\beta)$ has uniformly bounded length, it lies at a uniformly bounded distance from $\mb{T}$. However $f(\beta)$ lies in the middle of the product region $U$ where ${\rm inj}(U)\ge\ep$, therefore, if $D_0$ is larger than $d_Q(f(\beta),\mb{T})+B$, the only possibility is that the geodesic representative of $f(\beta)$ has length at least $2\ep$ and is contained in $U$. 
\end{proof}

Using the above setup we can control the relative position of two different product regions:

\begin{lem}
\label{middle gluing block}
Fix $\delta>0$. Let $\ep(\delta)>0$ be the constant of Theorem \ref{bounded}. Let $D_0(\ep)>0$ be the constant of Lemma \ref{not short}. Let $A(\delta),D_1(D_0,\delta)$ be the constants of Lemma \ref{position}. For every $h>0$ and $D\ge D_1$ there exist $T>0$ and $R>0$ such that the following holds: Let $Q=Q(Y,Y')$ be a quasi-fuchsian manifold. Suppose that we have:
\begin{enumerate}
\item{Large distance in the curve graph $d_{\mc{C}}(\Upsilon(Y),\Upsilon(Y'))\ge T$.}
\item{A pair of product regions $V_1,V_2\subset\mc{CC}(Q)$ that satisfy the following properties: They have size $D$ and injectivity radius bounded from below by ${\rm inj}(V_1),{\rm inj}(V_2)\ge\ep$. They contain the images of 2-Lipschitz maps $f:Z\rightarrow V_1$ and $f':Z'\rightarrow V_2$ with $d_Q(f(Z),\partial V_1)$, $d_Q(f'(Z'),\partial V_2)\ge D_0$ where $Z,Z'\in\T$ lie in the $A$-neighbourhoods of initial and terminal segments of length $h$ of $[Y,Y']$, respectively.}
\end{enumerate}
Then $V_1$ and $V_2$ are disjoint and cobound a gluing block $Q_0\subset\mc{CC}(Q)$ for which $V_1$ and $V_2$ are respectively the bottom and top collars. Furthermore, the volume of $Q_0$ is bounded by ${\rm vol}(Q)-{\rm vol}(Q_0)\le R$.
\end{lem}

\begin{proof}
We first show that the product regions are disjoint. 

Consider shortest geodesics $\alpha:=\Upsilon(Y)$ and $\alpha':=\Upsilon(Y')$. By Corollary 7.18 of Brock-Bromberg \cite{BB11} we have that 
\[
d_Q(\partial_Y\mc{CC}(Q),\partial_{Y'}\mc{CC}(Q))\ge\frac{1}{C}d_{\mc{C}}(\alpha,\alpha')-C\ge T/C-C
\]
for some uniform $C=C(g)>0$.

We now estimate the distance of $V_1$ from $\partial_Y\mc{CC}(Q)$. Let $\beta$ be a shortest geodesic for $Z$. It has $L_Z(\beta)\le B$ for a uniform constant $B=B(g)>0$. Denote by $\beta^*$ the geodesic representative of $\beta$ in $Q$. By Lemma \ref{not short}, we have $l_Q(\beta)\ge 2\ep$ and $\beta^*\subset V_1$ so that $d_Q(\partial_Y\mc{CC}(Q),V_1)\le d_Q(\beta^*,\partial_Y\mc{CC}(Q))$. By standard hyperbolic geometry 
\[
\cosh\left(d_Q(\beta^*,\partial_Y\mc{CC}(Q))\right)\le\frac{L_{\partial_Y\mc{CC}(Q)}(\beta)}{l_Q(\beta)}.
\]
Using the fact that $\partial_Y\mc{CC}(Q)$ and $Y$ are $K$-bilipschitz for some universal $K>0$, Wolpert's inequality $L_Y(\beta)\le e^{2d_\T(Y,Z)}L_Z(\beta)$, the assumptions $L_Z(\beta)\le B$ and $d_\T(Y,Z)\le h+A$, we get
\[
\frac{L_{\partial_Y\mc{CC}(Q)}(\beta)}{l_Q(\beta)}\le\frac{KB}{2\ep}e^{2(A+h)}.
\]
Similarly, we obtain an analogue estimate for the distance $d_Q(V_2,\partial_{Y'}\mc{CC}(Q))$.

Thus, as $V_1$ and $V_2$ have diameter at most $2D$ and are uniformly close to $\partial_Y\mc{CC}(Q)$ and $\partial_{Y'}\mc{CC}(Q)$ respectively, only depending on $h$, in order to make sure that they are disjoint, it is enough to require that $T$ is much larger than the previous constants. 

Let $Q_0$ be the gluing block bounded by $V_1,V_2$. From the previous discussion we also know that $V_1$ and $V_2$ are the bottom and top collars of $Q_0$ respectively.

We now estimate the volume of $Q_0$. We show that the region (homologically) bounded by $\partial_Y\mc{CC}(Q)$ and $f(Z)$ can be covered by a controlled number of straight hyperbolic simplices (recall that there is a universal upper bound on the volume of a hyperbolic straight simplex, see Chapter C of \cite{BP92}). To this purpose we need the following:

\begin{lem}
\label{triangulation}
Given two {\em complete clean markings} $\mu,\nu$ of $\Sigma$, it is possible to find a triangulation of $\Sigma\times[0,1]$ such that: 
\begin{itemize}
\item{The number of simplices is bounded by a uniform multiple of the distance between $\mu,\nu$ in the {\em marking graph}.}
\item{The induced triangulation to $\Sigma\times\{0\}$ and $\Sigma\times\{1\}$ can be taken to be any refinements of $\mu$ and $\nu$ to triangulations with the same vertex sets.}
\end{itemize} 
\end{lem} 

For the definition of {\em complete clean marking} and {\em marking graph} we refer to \cite{MM99}. We observe that if $X\in\T_\delta$, then there is a shortest complete clean marking $\mu_X$ of uniformly bounded length and, furthermore, if $[X,Y]\subset\T_\delta$ then the distance in the marking graph between the shortest complete clean markings of $X,Y$ is coarsely comparable with $d_\T(X,Y)$ by a result of Rafi \cite{R07}.  

\begin{proof}
Notice that every {\em complete clean marking} $\mu$ of $\Sigma$ decomposes the surface into a uniformly bounded number of polygonal disks with vertices in the intersections of the curves in the complete clean marking. We can refine this polygonal structure on $\Sigma$ to a triangulation by adding diagonals and keeping the same vertex set. Note that this procedure is by no means unique, but the number of combinatorial possibilities is bounded from above by a constant only depending on the genus of $\Sigma$.

Consider a geodesic sequence of elementary moves $\mu_1\rightarrow\cdots\rightarrow\mu_n$ in the marking graph that connects $\mu_1=\mu$ with $\mu_n=\nu$. We produce a triangulation of $\Sigma\times[0,n]$ by inductively stacking triangulations of $\Sigma\times[j-1,j]$ with a uniformly bounded number of simplices and such that the restriction of the triangulation to the boundaries $\Sigma\times\{j-1\}$ and $\Sigma\times\{j\}$ consists of triangulations $T_{j-1}$ and $T_j$ of $\Sigma$ which refine $\mu_{j-1}$ and $\mu_j$.

Let us assume that $\mu_j$ is obtained from $\mu_j$ by a {\em Dehn twist move} about the pants curves of $\mu_{j-1}$. Let $T_{j-1}$ be a triangulation of $\Sigma$ defined by $\mu_{j-1}$ and let $T_j$ be its image under the Dehn twist. Then there exists a triangulation $\tau_j$ of $\Sigma\times[j-1,j]$ which restricts to $T_{j-1},T_j$ on the boundary. As up to the action of the mapping class group there are only finitely many combinatorial possibilities for this situation, we can find such a triangulation of $\Sigma\times[j-1,j]$ with a uniformly bounded number of simplices. The same argument holds true for the move which replaces a pants curve by a marking curve and clears intersections. 
\end{proof}

We now return to the proof of the volume estimate in Lemma \ref{middle gluing block}.

Using Lemma \ref{triangulation}, we associate to the two surfaces $\partial_Y\mc{CC}(Q)$ and $Z$ a triangulation of $\Sigma\times[0,1]$: Let $\mu_Y$ and $\mu_Z$ be short complete clean markings for $\partial_Y\mc{CC}(Q)$ and $Z$ respectively. Notice that, since both hyperbolic surfaces lie in a uniformly thick part of Teichm\"uller space, we can find refinements of $\mu_Y$ and $\mu_Z$ to triangulations with uniformly bounded length. Let $H:\Sigma\times[0,1]\rightarrow Q$ be an arbitrary homotopy between the inclusion of $\partial_Y\mc{CC}(Q)$ and $f:Z\rightarrow Q$. 

We now {\em straighten} $H$ relative to the vertices (see Chapter C of \cite{BP92}). Since the boundary triangulations have uniformly bounded length on $\partial_Y\mc{CC}(Q)$ and $Z$, after straightening relative to the vertices their images of their 1-skeleta are contained in a uniform neighbourhood of $\partial_Y\mc{CC}(Q)$ and $f(Z)$. Since every point in a straight 2-simplex is uniformly close to the sides of the 2-simplex, we conclude that the whole straightening of $H$ restricted to the boundary lies uniformly close to $\partial_Y\mc{CC}(Q)$ and $f(Z)$. In particular, as $f(Z)$ lies deep inside $V_1$, the same is true for the straightening of $H(\Sigma\times\{1\})$. 

As a consequence, after straightening, $H$ still covers the region between $V_1$ and $\partial_Y\mc{CC}(Q)$ perhaps only missing a uniform neighbourhood of $\partial_Y\mc{CC}(Q)$ (which has uniformly bounded volume). The volume of the image of the straightened $H$ is bounded by a uniform constant (the maximal volume of a straight 3-simplex) times the number of simplices in the triangulation. By Lemma \ref{triangulation}, the number of simplices is bounded by the distance in the marking graph between $\mu_Y$ and $\mu_X$ which is coarsely equal to $d_\T(\partial_Y\mc{CC}(Y),Z)$ which in turn is coarsely bounded from above by $h$. 
\end{proof}

Notice that, since $\Upsilon[Y,Y']$ is a uniform unparametrized quasi-geodesic that restricts to a parametrized quasi-geodesic on $\delta$-thick subsegments (see \cite{H10}), there is a $T_1=T_1(T,\delta)>0$ such that if $[Y,Y']$ contains a $\delta$-thick subsegment of length $T_1$, then $d_\mc{C}(\Upsilon(Y),\Upsilon(Y'))\ge T$ and, hence, condition (1) is satisfied.

\subsection{A gluing theorem}

Finally, we are ready to state the gluing theorem in the form that we will use in the next sections:

\begin{thm}
\label{gluing}
Let $\delta,\xi>0$ be fixed. There exists $h_{\text{\rm gluing}}(\delta,\xi)>0$ such that for every $h\ge h_{\text{\rm gluing}}$ the following holds: Let $f$ be a gluing map. Consider a geodesic segment $[Y,Y']\subset\T$ with endpoints $Y,Y'\in\T_\delta$. Suppose that there exist $Y<X<X'<Y'$ satisfying the following relative bounded combinatorics and large heights properties:
\begin{itemize}
\item{We have $[Y,X],[Y',X']\subset\T_\delta$ and $d_\T(Y,X),d_\T(Y',X')\in[h,2h]$.}
\item{The pair $(Y,X)$ satisfies 
\[
d_{\mc{C}}(\Upsilon(X),\mc{D})\ge d_{\mc{C}}(\Upsilon(Y),\mc{D})+d_{\mc{C}}(\Upsilon(X),\Upsilon(Y))-\frac{1}{\delta}.
\]
The same holds true for the pair $(f^{-1}Y',f^{-1}X')$.}
\item{We have $d_{\mc{C}}(\Upsilon(Y),\Upsilon(Y'))\ge h$.}
\end{itemize}
Consider $N_1=H(X),N_2=H(f^{-1}X'),Q=Q(Y,Y')$. Then there exist: 
\begin{itemize}
\item{Disjoint product regions $V_1,V_2\subset\mc{CC}(Q)$ with $V_1$ below $V_2$ and both of size $D_1$ (as in Lemma \ref{middle gluing block}). We denote by $Q_0\subset\mc{CC}(Q)$ the gluing block bounded by $V_1,V_2$ for which $V_1$ and $V_2$ are respectively the bottom and top boundary.}
\item{Product regions $U_j\subset\mc{CC}(N_j)$ of size $D_1$ for $j=1,2$. We denote by $N_0^j\subset\mc{CC}(N_j)$ the gluing blocks that they bound.}
\item{Orientation preserving $\xi$-almost isometric diffeomorphisms $k_j:V_j\rightarrow U_j$ for $j=1,2$ where $k_1$ is in the homotopy class of the identity while $k_2$ in the homotopy class of $f$.}
\end{itemize}
In particular, we can form the 3-manifold
\[
X_f=N_0^1\cup_{k_1:V_1\rightarrow U_1}Q_0\cup_{k_2:V_2\rightarrow U_2}N_0^2
\] 
using the cut and glue construction as in Lemma \ref{cut-and-glue}. For simplicity we denote by $\Omega:=V_1\sqcup V_2$ the union of the product regions along which we performed the gluing. The manifold $X_f$ is diffeomorphic to $M_f=H_g\cup_fH_g$ and comes equipped with a Riemannian metric $\rho$ with the following properties:
\begin{enumerate}
\item{The sectional curvature of the metric is contained in the interval $(-1-\xi,-1+\xi)$ and it is constant $-1$ on $X_f-\Omega$.}
\item{Each connected component of $\Omega$ has uniformly bounded diameter and injectivity radius.} 
\item{The I-bundle piece $Q_0$ is isometric to the complement in $\mc{CC}(Q)$ of a collar neighborhood of the boundary of the convex core whose volume is uniformly bounded (only depending on $h$).}
\item{The handlebody pieces $N_0^1,N_0^2$ are isometric to the complement in $\mc{CC}(N_1),\mc{CC}(N_2)$ of collar neighborhoods of the boundaries of the convex cores of uniformly bounded diameter (not depending on $h$).} 
\end{enumerate}
\end{thm}

The proof is an application of the cut and glue construction Lemma \ref{cut-and-glue} where the input product regions $V_j,U_j$, handlebody gluing blocks $N_0^j$ and $\xi$-almost isometric diffeomorphisms are provided by Corollary \ref{cor unif emb} (applied to $[Y,X]\subset[Y,Y']$ and $f^{-1}[Y',X']\subset f^{-1}[Y',Y]$) and the I-bundle gluing block $Q_0$ is obtained by Lemma \ref{middle gluing block} (which also gives the volume estimate).


\section{Random Heegaard splittings}
\label{randomheegaard}

In this section we establish some geometric control on random 3-manifolds. 

We briefly recall the basic setup: Let $\mu$ by a {\em symmetric} probability measure on ${\rm Mod}(\Sigma)$ whose support $S$ generates the group. Let $\{s_j\}_{j\in\mb{N}}$ be a sequence of independent, $\mu$-distributed random variables with values in ${\rm Mod}(\Sigma)$. The $n$-th step of the {\em random walk driven by $\mu$} is the random variable $\omega_n:=s_1\cdots s_n$. We denote by $\mb{P}_n$ the distribution of $\omega_n$ (it coincides with the $n$-th convolution of $\mu$ with itself). If $\mc{P}_n$ is a property of 3-manifolds (possibly depending on the step of the walk) we say that $\mc{P}_n$ holds for a {\em random 3-manifold} if 
\[
\mb{P}_n\left[f\in{\rm Mod}(\Sigma)\left|\text{ $M_f$ has $\mc{P}_n$}\right.\right]\stackrel{n\rightarrow\infty}{\longrightarrow}1.
\]

With this notations, we can state the main result of the section: 

\begin{pro}
\label{random relative bounded}
Let $g\ge 2$ and $\ep,b,\delta>0$ be fixed. Let $\mu$ be a symmetric probability measure on 
${\rm Mod}(\Sigma)$ whose support is a finite generating set. Then
\[
\mb{P}_n\left[f\in{\rm Mod}(\Sigma)\left|\text{ $M_f$ has gluing with {\rm $(\epsilon,b,\delta)$-controlled geometry}}\right.\right]\stackrel{n\rightarrow\infty}{\longrightarrow}1.
\]
\end{pro}

We now explain what we mean by {\em $(\ep,b,\delta)$-controlled geometry}. We begin with the following definition:

\begin{dfn}\label{productregion}
For $\delta\in (0,1/2),b>1$ and $g\geq 2$, a {\em $(b,\delta)$-product region of genus $g$} in a Riemannian 3-manifold $M$ is a closed subset $V$ of $M$ with the following properties:
\begin{enumerate}
\item $V$ is diffeomorphic to $\Sigma\times [0,1]$ where $\Sigma$ is a closed surface of genus $g$, and $V$ {\em separates} $M$, that is, $M-{\rm int}(V)$ consists of two connected components with boundary $\Sigma\times \{0\},\Sigma\times \{1\}$, respectively. 
\item The injectivity radius of $M$ at points in $V$ is at least $\delta$, and the diameters of the surfaces $\Sigma\times \{0\}$ and $\Sigma\times \{1\}$ are at most $1/\delta$. 
\item The restriction of the metric of $M$ to $V$ is of constant curvature $-1$.
\item The distance between the boundary components $\Sigma\times \{0\}$ and $\Sigma\times \{1\}$ equals at least $b$.
\end{enumerate}
\end{dfn}  

Note that as, $b>1$, the volume of a $(b,\delta)$-product region is bounded from below by a universal constant. 

By definition, a $(b,\delta)$-product region $V\subset M$ separates $M$. In particular, if $V^\prime\subset M$ is another such region which is disjoint from $V$, then it is contained in one of the two components of $M-V$. Thus, if ${\mc V}\subset M$ is a disjoint union of $k\geq 1$ $(b,\delta)$-product regions, then the dual graph whose vertices are the components of $M-{\mc{V}}$ and where two such components are connected by an edge if their closures intersect the same component of ${\mc V}$ is a tree. We say that the components of ${\mc V}$ are {\em linearly aligned} if this tree is just a segment. 

The geometric control that we need is the following property $\mc{P}_n$ (depending on the step of the walk): 

\begin{dfn}[$(\ep,b,\delta)$-controlled geometry]
Let $n$ be the step of the random walk driven by $\mu$. We say that $M_f$ has a gluing with {\em $(\ep,b,\delta)$-controlled geometry} when the following properties are satisfied: There exist constants $C_1=C_1(\mu)>0$ and $C_2=C_2(b,\delta)>0$ such that the splitting $M_f$ admits a negatively curved metric as described in Theorem \ref{gluing} with the following additional features:
\begin{enumerate}
\item[(a)] The gluing control parameter $\xi$ is smaller than $\ep$. 
\item[(b)] ${\rm vol}(N_1^0\cup N_2^0\cup \Omega)\leq \ep n$ where $n$ is the step of the walk. 
\item[(c)] ${\rm vol}(Q_0)\geq C_1n$ where $n$ is the step of the walk. 
\item[(d)] The set $Q_0$ contains a subset $Q_0^\prime$ which is a disjoint union of $(b,\delta)$-product regions of genus $g$ and cardinality at least $C_2 n$.
\end{enumerate}
\end{dfn}

Notice that the constants $C_1,C_2>0$ implicitly contained in the statement of Proposition \ref{random relative bounded} are independent of $\ep$: The constant $C_1$ only depends on $\mu$ and the constant $C_2$ only depends on $b,\delta$. They will be determined in the course of the proof.

\subsection{Quasi-fuchsian manifolds with many $(b,\delta)$-product regions}
For us, the main example of a manifold with many linearly aligned $(b,\delta)$-product regions is the following:

\begin{lem}
\label{many product regions}
For every $b,\delta>0$ there exists $R>0$ such that the following holds: Consider a Teichm\"uller geodesic $\gamma:[0,T]\rightarrow\T$ with endpoints $Y:=\gamma(0)$ and $Y':=\gamma(T)$. Consider a sequence of times $R<t_1<\cdots<t_k<T-R$. Suppose that we have the following properties: The intervals $[t_j-R,t_j+R]$ are disjoint and their images $\gamma[t_j-R,t_j+R]$ are contained in $\T_\delta$. Then the convex core of $Q=Q(Y,Y')$ contains a collection of disjoint linearly aligned $(b,\delta)$-product regions $U_j\subset\mc{CC}(Q)$.
\end{lem}

\begin{proof}
Let $R>0$ be a very large constant. Consider $Q_j:=Q(\gamma(t_j-R/2),\gamma(t_j+R/2))$. By Corollary 7.18 of \cite{BB11} combined with \cite{H10}, we have that the width of $\mc{CC}(Q_j)$ is coarsely bounded from below by $R$. Moreover, by Theorem \ref{bounded}, the injectivity radius of the same manifold is bounded from below by $\ep=\ep(\delta)>0$. 

Thus, we can apply Lemma \ref{position} and Lemma \ref{large implies product}, and find a $(b,\ep)$-product region $U_j\subset\mc{CC}(Q_j)$ of uniformly bounded size that contains the image of a 1-Lipschitz map $f_j:Z_j\rightarrow U_j$, where $Z_j$ lies in a uniform neighbourhood of $\gamma[t_j-R/2,t_j+R/2]$, and such that $d_{Q_j}(U_j,\partial\mc{CC}(Q_j))$ is coarsely bounded from below by $R/2$. 

By Proposition \ref{unif emb}, if $R$ is large enough, we have a $\xi$-almost isometric embedding $k_j:U_j\rightarrow V_j\subset\mc{CC}(Q)$ in the homotopy class of the identity. By Lemma \ref{not short}, $V_j$ contains the geodesic representative $\alpha_j^*$ of $\Upsilon(Z_j)$, a curve which has length $l_{Q_j}(\alpha_j)$ in the interval $[2\ep,2B]$ for some uniform $B>0$. Thus, $d_Q(V_i,V_j)$ is coarsely bounded from below by $d_Q(\alpha_i^*,\alpha_j^*)$. 

By Theorem 7.16 of \cite{BB11}, the latter is coarsely bounded from below by $d_{\mc{C}}(\Upsilon(Z_i),\Upsilon(Z_j))$, which, in turn, by \cite{H10}, is coarsely bounded from below by $R$ (recall that $[t_i-R,t_i+R],[t_j-R,t_j+R]$ are disjoint subinterval of the same Teichm\"uller geodesic $[Y,Y']$ entirely contained in $\T_\delta$). Therefore, provided that $R$ is large enough also compared to the uniform size of the product regions $V_j$, we have that $V_1,\cdots,V_k$ are disjoint and, hence, linearly aligned (because each $\partial V_j$ is parallel to $\partial\mc{CC}(Q)$). 
\end{proof}

We remark that, with a little more effort, using the distance estimates provided by Theorem 7.16 of \cite{BB11} one can also establish that the linear order of the product regions is $V_1<\cdots<V_k$.

\subsection{Random walks on the mapping class group} \label{randommap}
We now recall some facts about random walks on ${\rm Mod}(\Sigma)$. 

\begin{center}
\label{notation3}
\begin{minipage}{.8\linewidth}
{\bf Standing assumptions}. In the sequel we always consider {\em symmetric} probability measures $\mu$, that is $\mu(s)=\mu(s^{-1})$, whose support $S$ is a {\em finite generating set} of the mapping class group ${\rm Mod}(\Sigma)$. 
\end{minipage}
\end{center}

Associated to the random walk generated by $\mu$ is a {\em space of sample paths} $\Omega:={\rm Mod}(\Sigma)^{\mb{N}}$ endowed with the $\sigma$-algebra of cylinder sets and the probability measure $\mb{P}:=T_*\mu^{\otimes \mathbb{N}}$ where $T:\Omega\rightarrow\Omega$ is the measurable map defined by $T(s_i)_{i\in\mb{N}}=(\omega_j:=s_1\cdots s_j)_{j\in\mb{N}}$.

We will use a geometric statement for the action of random mapping classes on Teichm\"uller space. The following
result is due to Tiozzo.

\begin{thm}[Tiozzo, Theorem 1 of \cite{T15}]
\label{tiozzo}
Fix some basepoint $o\in\T$ in the Teichm\"uller space of $\Sigma$. Then there exists $L_\T>0$ such that for almost all sample paths $\omega$ there exists a Teichm\"uller geodesic ray $\gamma_\omega:[0,\infty)\to\T$ with $\gamma(0)=o$ and such that
\[
\lim_{n\to \infty} \frac{d_{\mc{T}}(\omega_no,\gamma_\omega(L_\T n))}{n}\to 0.
\]
\end{thm}

Positivity of the drift $L_\T$ is a consequence of work of Kaimanovich and Masur \cite{KM96}.

There also is a statement concerning the action of the random walk on the curve graph $(\mc{C},d_{\mc{C}})$ of $\Sigma$ which is due to Maher and Tiozzo \cite{MT18}. 

\begin{thm}[Maher-Tiozzo, Theorem 1.2 and Theorem 1.3 of \cite{MT18}]
\label{mahertiozzo}
Let $\alpha\in \mc{C}$ be a basepoint. Then there exists $L_{\mc{C}}>0$ such that for almost every sample path $\omega=(\omega_n)_{n\in\mb{N}}$ we have
\[
\lim_{n\to \infty} \frac{d_{\mc{C}}(\alpha,\omega_n\alpha)}{n}=L_{\mc{C}}>0.
\]
Moreover, for almost every sample path $\omega$, there exists a uniformly quasi-geodesic ray $\eta_\omega\subset {\mathcal C}$ which tracks the sample path sublinearly, that is, 
\[
\lim_{n\to \infty}\frac{d_{\mc{C}}(\omega_n\alpha,\eta_\omega)}{n}=0.
\]
\end{thm}

As a combination of Theorem \ref{tiozzo}, Theorem \ref{mahertiozzo} and \cite{KM96} we have the following statement. For its formulation, recall that a point in the Gromov boundary $\partial\mc{C}$ of $\mc{C}$ is an unmeasured filling geodesic lamination on $\Sigma$.

\begin{thm}[\cite{KM96}, \cite{T15}, \cite{MT18}]
\label{exit map gromov}
For $\mb{P}$-almost every sample path $\omega\in\Omega$, the following holds true.
\begin{enumerate}
\item
  For every base-point $\alpha\in\mc{C}$,
  the sequence $\{\omega_n\alpha\}_{n\in\mb{N}}\subset\mc{C}$
  converges to a point 
  $\lambda_\omega\in\partial\mc{C}$
in the Gromov boundary of ${\mc{C}}$ which is
  independent of $\alpha$.
\item The point $\lambda_\omega$ supports a unique
  transverse invariant measure up to scale, and the Teichm\"uller ray
  $\tau_{o,\lambda_\omega}$ 
  issuing from a fixed basepoint $o\in \mc{T}$ which is determined
  by
  $\lambda_\omega$, equipped with this transverse invariant measure, 
  has the sublinear tracking property from Theorem \ref{tiozzo}.
\end{enumerate}

\end{thm} 

The next statement is Proposition 6.10 of \cite{BGH16}. 

\begin{pro}\label{prop610}
  Let $W\subset \mc{T}$ be a
  ${\rm Mod}(\Sigma)$-invariant open subset that contains an axis
  of a pseudo-Anosov mapping class. Then for all $h>0$
  there exists $\hat c=\hat c(W,h)>0$ such that for almost
every sample path $\omega$, we have 
\[\lim\inf\frac{1}{T}\vert \{t\in [0,T]\mid 
\tau_{o,\lambda_\omega}[t-h,t+h]\subset W\}\vert >\hat c.\]
\end{pro}

The ${\rm Mod}(\Sigma)$-invariants sets $W$ we are going to
use in the sequel are the sets $\mc{T}_\delta$ for some
suitably chosen numbers $\delta >0$.

\subsection{Random walks on Teichm\"uller space}
In this subsection we consider the orbit map of a random
walk on ${\rm Mod}(\Sigma)$ on Teichm\"uller space ${\mc{T}}$.
We always assume that the random walk is 
generated by a symmetric
probability measure $\mu$ whose finite support generates
${\rm Mod}(\Sigma)$. The results in this section are small variations of 
statements available in the literature, adjusted to our need. 
We provide proofs whenever we did not find a fully fitting reference. 

We begin with some information on the Teichm\"uller geodesic
$\tau_{o,fo}$ connecting a fixed point $o\in {\mc{T}}$ to $fo$ 
for a random mapping class $f$. To this end we use
a statement which is similar to statements from
\cite{H10} and to Proposition 4.6 of \cite{DH18}.

\begin{lem}\label{thick}
Let $\delta >0$ and $m>0$, $k>0$. There exists 
$R_0(\delta,m,k)>0$ with the following property. Let $R\geq R_0(\delta,m,k)$ and let
$\eta:[-R,R]\to {\mc{T}}_\delta$ be a Teichm\"uller geodesic segment. Let furthermore  
$\gamma:\mathbb{R}\to {\mc{T}}$ be a Teichm\"uller geodesic 
whose projection to ${\mc{C}}$ contains the projection of $\eta$
in its $k$-neighborhood. Then $\eta(0)$ is contained in the 
$m$-neighborhood of $\gamma$.
\end{lem} 
\begin{proof}
Assume that the lemma does not hold. Then there is a number $\delta >0$, a number 
$k>0$  and  a number $m>0$, 
and there is a sequence $R_i\to \infty$ and a
sequence of counter-examples, given by geodesic arcs
$\eta_i:[-R_i,R_i]\to {\mc{T}}_\delta$ and geodesics $\gamma_i$ 
whose projections to ${\mc{C}}$ contain the projections of $\eta_i$ in their
$k$-neighborhood, 
but such that 
$d_{\mc{T}}(\eta_i(0),\gamma_i)\geq m$.

Since the mapping class group acts properly and cocompactly on
${\mc{T}}_\delta$, by invariance under the action of the mapping class group we may assume
that $\eta_i(0)$ is contained in a fixed compact subset of ${\mc{T}}_\delta$. Thus
by passing to a subsequence, we may
assume that $\eta_i\to \eta:\mathbb{R}\to {\mc{T}}_\delta$ $(i\to \infty)$.
Now the projections $\Upsilon(\eta_i)$ 
of the segments $\eta_i$ to the curve graph $\mc{C}$ are 
parameterized $p$-quasi-geodesics for a number $p>0$ only depending on $\delta$
\cite{H10}, and the same holds true for $\eta$.
By convergence, as $i\to \infty$ the segments $\eta_i$ fellow-travel
$\eta$ on longer and longer subsegments. Moreover, the projection $\Upsilon$ is coarsely 
Lipschitz and hence by hyperbolicity of the curve graph, 
the projections of the endpoints of $\eta_i$ to ${\mc{C}}$ 
converge to the endpoints of the quasi-geodesic $\Upsilon(\eta)$ in the Gromov
boundary $\partial\mc{C}$ of ${\mc{C}}$. 
These endpoints are uniquely ergodic filling measured geodesic laminations \cite{H10}.

On the other hand, up to parameterization, 
the geodesics $\gamma_i$ are determined by a pair of points in the space
$\mc{PML}$ of projective measured laminations. 
By slightly changing $\gamma_i$ without changing the property
that the projection of $\eta_i$ is contained in the $k$-neighborhood of the projection of 
$\gamma_i$, we may assume that these endpoints are filling measured geodesic laminations
and hence they define a pair of points in $\partial\mc{C}$.  
This pair of points can be connected by a uniform quasi-geodesic in ${\mc{C}}$ 
which contains $\Upsilon(\eta_i)$ 
in its $k$-neighborhood. As the distance between $\Upsilon(\eta_i(0))$ and each of the two 
endpoints $\Upsilon(\eta_i(\pm R_i))$ of $\Upsilon(\eta_i)$ 
tends to infinity with $i$, 
hyperbolicity of the curve graph implies that the endpoints of 
the quasi-geodesics $\Upsilon(\gamma_i)$ converge as $i\to \infty$ to 
the endpoints of $\Upsilon(\eta)$.

Since the endpoints of $\Upsilon(\eta)$ in $\partial {\mc{C}}$ are 
uniquely ergodic, by the properties of the topology on $\partial {\mc{C}}$ \cite{K99,H06},
we conclude that the points in ${\mc{PML}}\times {\mc{PML}}$ 
which determine $\gamma_i$ converge in ${\mc{PML}}\times {\mc{PML}}$ to the  
point which determines $\eta$. 
By continuity of the dependence of a Teichm\"uller geodesic on the pair of its vertical and horizontal 
measured geodesic lamination, this implies that $\gamma_i\to \eta$. Thus by 
continuity, we have $d_{\mc{T}}(\eta_i(0),\gamma_i)\to 0$. This is a contradiction which
proves the lemma.
\end{proof}

Lemma \ref{thick} can be applied to obtain a fellow traveling
statement for Teichm\"uller geodesics.

\begin{dfn}\label{weakfellow}
For numbers $R>0,m>0$ and a sufficiently small number $\delta>0$, a 
Teichm\"uller geodesic segment $\eta:[0,T]\to {\mc{T}}$ {\em weakly
$(R,m,\delta)$-fellow
travels} a (finite or infinite) Teichm\"uller geodesic $\gamma$   
if the following holds true. Assume that 
$[s,t]\subset [0,T]$ 
is such that $t-s\geq 2R$ and 
$\eta([s,s+2R]\cup [t-2R,t]) \subset {\mc{T}}_\delta$, 
then $\eta[s+R,t-R]$ is contained in the 
$m$-neighborhood of $\gamma$.
\end{dfn}

Note that the definition is not symmetric in $\gamma,\eta$, and this fact will be convenient for us. 

As a consequence of Lemma \ref{thick}, the main result of \cite{H10} and results of 
\cite{R14} we obtain

\begin{lem}\label{weakfellowlem}
For every $k>0,\delta>0$ there are numbers $R=R_1(k,\delta)>0, m=m(k,\delta)>0$  
with the following property. 
Let $\gamma:(a,b)\to {\mc{T}}$ be a finite or infinite 
Teichm\"uller geodesic and let $\eta:[0,T]\to {\mc{T}}$ be a Teichm\"uller
geodesic segment such that $\Upsilon(\eta[0,T])$ is contained in the 
$k$-neighborhood of $\Upsilon(\gamma(a,b))\subset {\mc{C}}$. Then 
$\eta$ weakly $(R,m,\delta)$-fellow travels $\gamma$.
\end{lem}
\begin{proof} 
By the main result of \cite{H10}, for a number $\delta >0$, 
the image under the projection $\Upsilon$ of 
a sufficiently long Teichm\"uller geodesic segment $\gamma$ 
is a uniform quasi-geodesic in the 
curve graph if and only if $\gamma$ is entirely contained in the thick part of 
Teichm\"uller space, 
and this statement
can be made quantitative, relating the thickness constant to the
quasi-geodesic constant. 

Since the image under the projection $\Upsilon$ of any Teichm\"uller geodesic is 
a uniform
{\em unparameterized} quasi-geodesic in the curve graph, we conclude that the 
projection  to the curve graph of a sufficiently long Teichm\"uller geodesic segment 
detects when the segment is contained in ${\mc{T}}_\delta$ for an a priori chosen
number $\delta>0$.

Lemma \ref{thick} then states that up to a constant which depends on $\delta$, the precise location of such a sufficiently long 
Teichm\"uller geodesic segment in $\T_\delta$ is determined by its projection into the curve graph.

This completes the proof of the lemma as follows. If $\Upsilon(\eta[0,T])$ is contained
in the $k$-neighborhood of  
$\Upsilon(\gamma)$, then the projection to the curve graph of a sufficiently
long subsegment of $\eta$ contained in ${\mc{T}}_\delta$ is a uniform parameterized
quasi-geodesic which 
uniformly fellow travels
the projection of some segment of $\gamma$. Using Lemma \ref{thick}, we conclude that
under the assumption of the lemma, for given $\delta >0,k>0$ there 
is a number $R=R(k,\delta)>0$ such that
if $[s,s+2R]\subset [0,T]$ and $\eta[s,s+2R]\subset {\mc{T}}_\delta$, then 
$\eta(s+R)$ is contained in the 1-neighborhood of $\gamma$.

Now the main result of \cite{R14} implies that two Teichm\"uller geodesic segments whose
endpoints are of distance at most $1$ and 
contained in ${\mc{T}}_\delta$ 
are uniform fellow travelers, where the uniformity constant $m=m(1,\delta)$ 
depends on $\delta$. 
This completes the proof the lemma.
\end{proof}

In the statement of the following proposition, for a subset $A$ of $\mc{T}$ and a number
$m>0$, we denote by 
$N_m(A)$ the neighborhood of radius $m$ about $A$ for the Teichm\"uller metric. 
The following result is 
the main technical result for random
walks on Teichm\"uller space needed towards our goal. In its formulation, $L_{\mc{T}}$ denotes
as before the drift of the random walk on Teichm\"uller space.

\begin{pro}\label{recursion}
Let $\mu$ be a symmetric probability measure on ${\rm Mod}(S)$ whose support is a finite
generating set. There exists a number $m>0$ such that for all $\epsilon >0$, we have
\[\mb{P}[\omega\in\Omega\mid \tau_{o,\omega_no}[0,nL_{\mc{T}}(1-\ep)]\subset 
N_m(\tau_{o,\lambda_\omega})]\stackrel{n\to\infty}{\longrightarrow}1.\]
\end{pro}

\begin{proof}
Let $\delta >0$ be sufficiently small that ${\mc{T}}_{2\delta}$ contains the axis of a pseudo-Anosov 
element and let $o\in {\mc{T}}_{2\delta}$ by the basepoint for the random walk on ${\mc{T}}$. 
Consider the projection of the random walk to the curve graph via the systole map $\Upsilon$. Denote as
before by $L_{\mc{C}}>0$ the drift of this random walk. 

By hyperbolicity of the curve graph, 
there exists a number $k>0$ with the following property. 
Let $p>1$ be such that the projection of a Teichm\"uller geodesic to ${\mc{C}}$ is 
an unparameterized $p$-quasi-geodesic. Then if 
$\alpha,\beta:\mathbb{R}\to {\mc{C}}$ are two simplicial 
$p$-quasi-geodesics and if the diameter of the shortest distance projection of 
$\alpha$ into $\beta$ equals $q>0$,
then there is a subsegment of $\alpha$ of length at least $q-2k$ which is contained
in the $k$-neighborhood of $\beta$. 

For this number $k>0$ and for an arbitrarily chosen number $\sigma >0$, we obtain from 
Proposition 2.6 and Proposition 2.9 of \cite{DH18} that
we have  
\[
\mb{P}[\omega\in\Omega\mid 
\Upsilon(\tau_{o,\omega_no}[0,\rho(\omega_n)])\subset 
N_k(\Upsilon(\tau_{o,\lambda_\omega}[0,\infty))]\stackrel{n\to\infty}{\longrightarrow}1
\]
where
\[
\rho(\omega_n)=\sup\{t\mid d_{\mc{C}}(\Upsilon(o),
\Upsilon(\tau_{o,\lambda_\omega}(s))\leq (1-\sigma)nL_{\mc{C}}\,\forall s\leq t\}.
\]
Furthermore, we have
\[
\mb{P}[\omega\in\Omega\mid \lim_{T\to \infty} \frac{1}{T}d_{\mc{C}}(\Upsilon(o),
\Upsilon(\tau_{o,\lambda_\omega}(T)))=L_{\mc{C}}/L_{\mc{T}}]=1
\]
(this fact can be found in many references, for example, it is the identity
(13) in \cite{BGH16}). 
As a consequence, we obtain that
\[
\mb{P}[\omega\in\Omega\left|\rho(\omega_n)\geq (1-\sigma)^2 nL_{\mc{T}}\right.]\stackrel{n\to\infty}{\longrightarrow}1.
\]

Let $R=R_1(k,\delta)>0$ be as in 
Lemma \ref{weakfellowlem} and let $\ep >0$.
By Lemma \ref{weakfellowlem}, the assumption that $o\in {\mc{T}}_{2\delta}$
and by \cite{R14}, fellow traveling almost surely 
of $\tau_{o,\omega_no}$ and $\tau_{o,\lambda_\omega}$ on a segment
of length $(1-\ep)L_{\mc{T}}n$ is guaranteed if   
for $\sigma>0$ with $(1-\sigma)^2 >1-\ep/4$ we have
\begin{align*}
\mathbb{P}[\omega\in\Omega\mid &\text{ there is } t\in [(1-\ep)nL_{\mc{T}},(1-\sigma)^2nL_{\mc{T}}] \\
&\text{ such that } \tau_{o,\lambda_\omega}[t-R,t+R]\subset {\mc{T}}_\delta]\stackrel{n\to\infty}{\longrightarrow}1.
\end{align*}
That this holds indeed true is shown in Proposition 6.14 of \cite{BGH16}. 
\end{proof}

\subsection{Random handlebodies}

From now on we fix a handlebody $H_g$ and a marking of the boundary surface $\Sigma$. Let $\mc{D}$ be corresponding disk set. Maher, exploiting work of Kerckhoff \cite{Ker90} (see also Gadre \cite{Ga12}), showed the following:

\begin{thm}[Maher, \cite{Ma10}]
\label{disk measure zero}
The {\rm Hempel distance} increases linearly along the random walk, that is, there exists a constant $K>1$ such that
\[
\mb{P}\left[\omega\in\Omega\mid d_{\mc{C}}(\mc{D},\omega_n\mc{D})\in[n/K,Kn]\right]\stackrel{n\rightarrow\infty}{\longrightarrow}1.
\]
\end{thm}

Furthermore, the convergence in Theorem \ref{disk measure zero} happens exponentially fast as proved by Lubotzky-Maher-Wu \cite{LMW16}.

Maher's theorem has a few immediate consequences. First of all, for a random mapping class $\omega_n$, the 3-manifold $M_{\omega_n}$ is hyperbolic (see Hempel \cite{He01} and Dunfield-Thurston \cite{DT06}). Furthermore, let us choose once and for all a basepoint $o\in \mc{T}_\delta$ contained in the $\delta$-thick part of Teichm\"uller space for a suitably chosen number $\delta>0$. We select $o$ so that it admits a short complete clean marking whose base is a pants decomposition made of diskbounding curves for $H_g$. By Theorem \ref{mahertiozzo}, the distance in the curve graph between $\Upsilon(o)$ and $\Upsilon(\omega_no)$ makes linear progress in $n$, and by Theorem \ref{disk measure zero}, it makes linear progress away from the diskbounding curves. Here as before, $\Upsilon:\mc{T}\to \mc{C}$ denotes the systole map.

This property, however, is not sufficient to conclude that for a random element $\omega_n\in {\rm Mod}(\Sigma)$, the manifold
$M_{\omega_n}$ satisfies the assumptions in Proposition \ref{random relative bounded}. As additional properties, we have to control the transition of the Teichm\"uller geodesic segment $\tau_{o,\omega_no}$ connecting $o$ to $\omega_no$ through the thick part of Teichm\"uller space while controlling the rate of divergence of its trace from the disk set. We next establish this control.

Thus let $\omega_n\in{\rm Mod}(\Sigma)$ be a random mapping class. 

The subset $G_{\omega_n}:=\Upsilon (\tau_{o,\omega_no})$ of the curve graph ${\mc{C}}$ is a uniform unparameterized quasi-geodesic. Denote by $\pi_{G_{\omega_n}}$ the nearest point projection of the curve graph ${\mc{C}}$ onto $G_{\omega_n}$. Hyperbolicity of $\mc{C}$ yields that the projection $\pi_{G_{\omega_n}}(\mc{D})$ is a quasi-convex subset of $G_{\omega_n}$, with control constants not depending on $\omega_n$. Let $\vert \pi_{G_{\omega_n}}(\mc{D})\vert$ be its diameter. Our next goal is to prove that as the step length tends to infinity, this diameter is arbitrarily small compared to the diameter $\vert G_{\omega_n}\vert$ of $G_{\omega_n}$. This implies that $\Upsilon\tau_{o,\omega_no}$ can only be close to $\mc{D}$ on a small initial segment.
 
\begin{pro}
\label{small fellow-travelling}
Let $g\ge 2$ and $\ep>0$ be fixed. Let $\mu$ be a symmetric probability measure on ${\rm Mod}(\Sigma)$ whose 
support is a finite generating set. We have
\[
\mb{P}[\omega\in\Omega\left| \text{We have $\vert\pi_{G_{\omega_n}}(\mc{D})\vert/\vert G_{\omega_n}\vert\leq\ep$} \right.]\stackrel{n\rightarrow\infty}{\longrightarrow}1.
\]
\end{pro}

\begin{proof} 
Let $K>0$ be the constant from Theorem \ref{disk measure zero} and let $L_{\mc{C}}$ be the constant from Theorem \ref{mahertiozzo}. Let $\ep>0$ be such that $L_{\mc{C}}\epsilon <1/2K$. 

Let $\alpha:=\Upsilon(o)\in {\mc{C}}$. By the choice of $o\in\T$, we may assume that $\alpha$ is diskbounding in the handlebody $H_g$.
   
For $n_0>0$ let $\Omega_{n_0}\subset\Omega$ be the set of all sample paths $\omega=(\omega_n)$ such that for all $n\geq n_0$  the following properties are fulfilled. 
\begin{enumerate}
\item{$L_{\mc{C}}(1-\epsilon/2)n\leq d_{\mc{C}}(\alpha,\omega_n\alpha)\leq L_{\mc{C}}(1+\epsilon/2)n$.}
\item{Let $\gamma_\omega$ be a uniform quasi-geodesic ray in $\mc{C}$ connecting $\gamma_\omega(0)=\alpha$ to $\gamma_\omega(\infty)=\lambda_\omega$. Then $d_{\mc{C}}(\gamma_\omega,\omega_n\alpha)\leq L_{\mc{C}}\epsilon n/2$.
\item $d_{\mc{C}}(\mc{D},\omega_n\mc{D})\geq n/2K$.}
\end{enumerate}
Note that we have $\Omega_{n_1}\supset\Omega_{n_0}$ for all $n_1\geq n_0$. By Theorem \ref{mahertiozzo} and Theorem \ref{disk measure zero}, for every $\rho>0$  there exists a number $n_0=n_0(\rho)>0$ so that $\mb{P}(\Omega_{n_0})\geq 1-\rho$ for all $n\geq n_0$.
 
The disk set $\mc{D}\subset \mc{C}$ is quasi-convex. Thus, by hyperbolicity of $\mc{C}$, there exists a number $A>0$ with the following property. Let $\zeta:[0,\infty)\to \mc{C}$ be a uniform quasi-geodesic ray beginning at $\zeta(0)=\alpha\in \mc{D}$. If $t>0$ is such that $d_{\mc{C}}(\zeta(t),\mc{D})>A$ and if $\beta\in \mc{C}$ is such that
$\zeta(t)$ equals a shortest distance projection of $\beta$ into $\zeta$, then a shortest geodesic connecting $\beta$ to $\mc{D}$ passes through a uniformly bounded neighborhood of $\zeta(t)$. In particular, up to increasing $A$, the point $\zeta(t)$ is not contained in the image of the shortest distance projection of ${\mc{D}}$ into the uniform quasi-geodesic ray $\zeta$.
  
Assume from now on that $n_0/4K>A$. Let $(\omega_n)\in \Omega_{n_0}$ and let $n\geq n_0$. Denote by $\gamma_\omega$ the quasi-geodesic ray in $\mc{C}$ as in property (2) above. Then, on the one hand, we have
\[
L_{\mc{C}}(1-\epsilon/2)n\leq d_{\mc{C}}(\alpha,\omega_n\alpha)\leq L_{\mc{C}}(1+\epsilon/2)n,
\]
on the other hand, also $d_{\mc{C}}(\gamma_\omega,\omega_n\alpha)\leq L_{\mc{C}}\epsilon n/2$. In particular, by property (3) above, the nearest point projection $q_n$ of $\omega_n\alpha$ into $\gamma_\omega$ is of distance at least $n/2K-L_{\mc{C}}\epsilon n/2\geq n/4K>A$ from $\mc{D}$. This implies that a geodesic in ${\mc{C}}$ which connects $\omega_n\alpha$ to a shortest distance projection into $\mc{D}$ passes through a uniformly bounded neighborhood of $q_n$. Using again uniform quasi-convexity of $\mc{D}$ and the fact that $\alpha\in\mc{D}$, we conclude that the diameter of the shortest distance projection of $\mc{D}$ into the uniform quasi-geodesic $\Upsilon(\tau_{o,\omega_no})$ does not exceed the distance between $\alpha$ and $q_{n_0}$. Hence this diameter is at most $L_{\mc{C}}(1+\epsilon)n_0$, independent of $n\geq n_0$ and $\omega\in\Omega_{n_0}$. 
 
Choose $n_1>0$ sufficiently large that $L_{\mc{C}}(1+\epsilon)n_0\leq\ep L_{\mc{C}}(1-\epsilon)n_1$. Then for $\omega\in \Omega_{n_0}$ and for $n\geq n_1$, the distance between $\omega_n\alpha$ and $\alpha$ is at least $L_{\mc{C}}(1-\ep /2)n$, while the diameter of the projection of $\mc{D}$ into $\Upsilon(\tau_{o,\omega_no})$ does not exceed $L_{\mc{C}}(1+\ep)n_0$. By the choice of $n_1$, this means that for $n\geq n_1$ and every $\omega\in\Omega_{n_0}$, the properties required in the proposition are fulfilled for $\omega_n$, that is, we have $\vert \pi_{G_{\omega_n}}(\mc{D})\vert \leq\ep\vert G_{\omega_n}\vert $ as claimed.

As $\rho>0$ was arbitrary, the proposition follows. 
\end{proof}

Let as before $A>0$ be sufficiently large that the following holds true. Let $\gamma:[0,\infty)\to {\mc{C}}$ be a uniform quasi-geodesic beginning at the diskbounding curve $\gamma(0)=\alpha$ (this should mean that we choose once and for all a quasi-geodesic constant so that any two distinct points in $\mc{C}\cup \partial\mc{C}$ can be connected by a quasi-geodesic for this constant). We require that whenever $\beta\in \mc{C}$ is such that a shortest distance projection $\gamma(t)$ of $\beta$ into $\gamma$ has distance at least $A$ from the set ${\mc{D}}$ of diskbounding curves in $H_g$, then a shortest geodesic connecting $\beta$ to $\mc{D}$ passes through a uniformly bounded neighborhood of $\gamma(t)$. That such a number $A>0$ exists was a main technical ingredient in the proof of Proposition \ref{small fellow-travelling}.

Consider again a symmetric probability measure $\mu$ on ${\rm Mod}(\Sigma)$ whose support is a finite generating set and which induces the probability measure $\mathbb{P}$ on $\Omega$. Let $\delta, R,\ep>0$ be arbitrarily fixed. We require that $\delta>0$ is small enough that the conditions in Proposition \ref{prop610} are fulfilled for $W={\mc{T}}_{2\delta}$. 

Let $f\in {\rm Mod}(\Sigma)$ and consider the Teichm\"uller geodesic $\tau_{o,fo}$ connecting $o$ to $fo$. We say that $\tau_{o,fo}$ is {\em $(R,\delta,\epsilon)$-admissible} if the following holds true. There exist numbers $\rho_1,\rho_2\leq 
\epsilon d_{\mc{T}}(o,fo)$ such that:
\begin{itemize}
\item{The distance between $\Upsilon(\tau_{o,fo}[\rho_1-2R,\rho_1])$ and $\mc{D}$ and the distance between 
$\Upsilon(\tau_{o,fo}[d_{\mc{T}}(o,fo)-\rho_2,d_{\mc{T}}(o,fo)-\rho_2+2R])$ and $f\mc{D}$ is at least $A$.}
\item{$\tau_{o,fo}([\rho_1-2R,\rho_1]\cup[d_{\mc{T}}(o,fo)-\rho_2,d_{\mc{T}}(o,fo)-\rho_2+2R])\subset \mc{T}_\delta$.} 
\end{itemize}

Thus if the Teichm\"uller geodesic segment $\tau_{o,fo}$ is admissible, then it contains a subsegment of length at least $2R$ 
which is contained in the $\delta$-thick part of Teichm\"uller space, whose projection to the curve graph is separated from the set of diskbounding curves in a controlled way, and which is located uniformly near the starting point of the geodesic. Furthermore, there also is a segment of length at least $2R$ with these properties near the endpoint of the geodesic segment.

\begin{pro}\label{admissible}
We have
  \[
  \mathbb{P}[\omega\in\Omega\mid\text{The segment $\tau_{o,\omega_no}$ is $(R,\delta,\epsilon)$-admissible }]\stackrel{n\to\infty}{\longrightarrow}1.
  \] 
\end{pro}

\begin{proof}   
Let $R>0$, let $\delta>0$ be sufficiently small but fixed, and let $\ep >0$. Let furthermore $\sigma >0$. For $n>0$ define 
\begin{align*}
\Omega_{n,0}& :=\{\omega\in\Omega\mid d_\T(o,\omega_no)\in [nL_{\mc{T}}(1-\ep/2), nL_{\mc{T}}(1+\ep/2)] \text{ and }\\
 & \tau_{o,\omega_no}[\ep nL_{\mc{T}}/2,\ep nL_{\mc{T}}]\text{ contains a segment of length }2R \text{ in }\T_\delta\}.
\end{align*}
Proposition \ref{recursion} together with Proposition 6.14 of \cite{BGH16} shows that 
\[
\mathbb{P}(\Omega_{n,0})\stackrel{n\to\infty}{\longrightarrow}1.
\]
On the other hand, if for a fixed number $A>0$, chosen as above, we define 
\[
\Omega_{n,1}:=\{\omega\in\Omega\mid d_\mc{C}(\Upsilon(\tau_{o,\omega_no}(\epsilon nL_{\mc{T}}/2)),\pi_{G_{\omega_n}}({\mc {D}}))\geq A\},
\] 
then Proposition \ref{small fellow-travelling} shows that
\[
\mathbb{P}(\Omega_{n,1})\stackrel{n\to\infty}{\longrightarrow}1.
\]

Now if $\omega\in \Omega_{n,0}\cap \Omega_{n,1}$, then there exists an initial segment of length $2R$ on the geodesic $\tau_{o,\omega_no}$ which is mapped into ${\mc{T}}_\delta$ as required in the definition of admissibility, and this segment is separated away from the projection of the disk sets into the projection of $\tau_{o,\omega_no}$. Thus, such an element fulfills the requirement in the definition of admissibility near the starting point of $\tau_{o,\omega_no}$. Reversal of time then implies that with probability tending to 1 as $n\to \infty$, we may assume that the same is true for the inverse $\tau_{\omega_no,o}$. Together this shows the proposition. 
\end{proof}

\subsection{Good gluing regions}

The goal of this subsection is to show
Proposition \ref{random relative bounded}.
The argument is very similar to the argument in the proof 
of Proposition \ref{small fellow-travelling}.

We begin with a volume control for convex cocompact
hyperbolic structures on handlebodies. 
To this end
choose as before  a marking $\eta$ for the boundary $\Sigma$ of the handlebody
$H_g$ so that the base pants decomposition consists of diskbounding curves.
The following proposition is well known in various settings. As we did not find a directly
quotable statement in the literature, we sketch a proof.

\begin{pro}\label{finitevolume}
Let $\epsilon>0$ be a fixed number and let $\nu$ be any complete clean marking on $\Sigma$ with the property that the distance in the curve graph of a component of $\nu$ to a diskbounding curve is at least three. Suppose that $H_g$ is equipped with a convex cocompact hyperbolic structure $H(X)$ with conformal boundary $X\in \mc{T}_\delta$ such that $\nu$ is short for $X$. Then the volume of the convex core of $H(X)$ is bounded from above by a fixed multiple of the distance between $\eta,\nu$ in the marking graph.  
\end{pro}

\begin{proof} 
We proceed as in the volume estimate in Lemma \ref{middle gluing block} and show that we can cover $\mc{CC}(H(X))$ with a straight triangulation where number of simplices is bounded from above by a fixed multiple of the distance between $\eta,\nu$ in the marking graph. 

Let us consider for the moment an abstract handlebody $H$ of genus $g$ and a complete clean marking $\hat \eta$ of $\partial H$ whose pants decomposition consists of diskbounding curves. The disks bounded by these curves decompose $H$ into balls. Thus, up to a diffeomorphism of $H$, there are only finitely many combinatorial possibilities for the marking $\hat \eta$. As a consequence, a triangulation of $\partial H$ constructed from $\hat \eta$ by adding diagonals can be extended to a triangulation of $H$ with uniformly few simplices. 

Now, the pants decomposition of the marking $\eta$ consists of diskbounding curves and hence there exists an embedded handlebody $H\subset \mc{CC}(H(X))$ with the following property.
\begin{itemize}
\item $\mc{CC}(H(X))-{\rm int}(H)$ is diffeomorphic to $\Sigma\times [0,1]$.
\item $H$ is triangulated into uniformly few simplicies, and the restriction of this triangulation to $\partial H=\Sigma\times \{0\}\subset \mc{CC}(H(X))$ is constructed from the marking which is the image of $\eta\subset \Sigma\times \{1\}=\partial\mc{CC}(H(X))$ by the diffeomorphism $\Sigma\times \{1\}\to \Sigma \times \{0\}$ isotopic to the inclusion.
\end{itemize}

By Lemma \ref{triangulation}, we can extend this triangulation to $\mc{CC}(H(X))-{\rm int}(H)$ in such a way that restriction to 
$\Sigma\times \{1\}$ is a subdivision of $\nu$ and the number of simplices is bounded from above by a fixed multiple of the distance between $\eta$ and $\nu$ in the marking graph. 

By the assumption that $\nu$ is short for $X$, the marking $\nu$ is short for the boundary of the convex core (by Theorem \ref{bridgeman-canary}). Thus straightening this triangulation yields a triangulation of a subset of $\mc{CC}(H(X))$ whose complement is contained in a uniformly bounded neighborhood of the boundary. Since the boundary has uniformly bounded diameter, such a neighbourhood has uniformly bounded volume. This yields the proposition.  
\end{proof}

We are now ready to complete the proof of Proposition \ref{random relative bounded}. 

\begin{proof}[Proof of Proposition \ref{random relative bounded}]
Let $o\in\T_\delta$ be a point in the thick part of Teichm\"uller space for which a fixed complete clean marking $\eta$ on $\Sigma$ with pants curves consisting of diskbounding curves in the handlebody $H_g$ is short. The strategy is to find a quadruple of points $Y<X<X'<Y'$ on the Teichm\"uller geodesic $\tau_{o,\omega_no}$ connecting $o$ to its image under a random mapping class $\omega_n$ which fulfills the assumptions in Theorem \ref{gluing}. Furthermore, the points $Y,X$ should be contained in the initial subsegment of the geodesic of length at most $\ep$ times the total length, and the points $X',Y'$ should be contained in the terminal subsegment of the geodesic of length at most $\ep$ times the total length.

Using Proposition \ref{finitevolume}, we then argue that the sum of the volumes of the convex cocompact handlebodies corresponding to this initial and terminal segments of the Teichm\"uller geodesic cut at $X,X'$ are small compared to the volume of the center piece. Using Lemma \ref{many product regions}, we will then show that the center piece contains linearly aligned product regions as predicted in the proposition.  

By Proposition \ref{admissible}, it suffices to assume that for an a priori given constant $R>0$, the Teichm\"uller geodesic $\tau_{o,\omega_no}$ is $(R,\delta,\ep)$-admissible. Let us assume that $\rho_1<\ep d_{\mc{T}}(o,\omega_no)$ and $\rho_2>(1-\ep)d_{{\mc{T}}(o,\omega_no)}$ are as in the definition of $(R,\delta,\ep)$-admissibility. Combining Proposition \ref{recursion}, Proposition \ref{prop610} and Proposition \ref{admissible}, we conclude that we may assume that the total measure of the closed set $A\subset [\rho_1,\rho_2]$ of all points $t$ such that $\tau_{o,\omega_no}[t-R,t+R]\subset {\mc{T}}_\delta$ is at least $\hat c d_{\mc{T}}(o,\omega_no)/2$.

Let $\lambda$ be the standard Lebesgue measure on $\mathbb{R}$ and for $m\geq 1$ define $t(m)>0$ by $t(m)=\sup\{t\in A\mid\lambda[\rho_1,t]\cap A<2mR\}$. Since $A$ is closed, we have $t(m)\in A$. Furthermore, as $\lambda$ is the standard Lebesgue measure, we also have $t(m-1)\leq t(m)-2R$. In particular, the open intervals $(t(m)-R,t(m)+R)$ are pairwise
disjoint, moreover there are at least $\hat c/4R$ such intervals. 

Consider the quasi-fuchsian manifold defined by the $(R,\delta,\epsilon)$-good gluing region and the Teichm\"uller segment
 $\tau_{o,\omega_no}$. Lemma \ref{many product regions} shows that this quasi-fuchsian manifold satisfies property (d) in the definition of a gluing with $(\epsilon,b,\delta)$-controlled geometry where the constant $R>0$ as above depends on the choice of the a priori prescribed number $b>1$. This completes the proof of the proposition.  
 \end{proof}

\section{Geometric control of random hyperbolic 3-manifolds}\label{geometric}

In Section \ref{randomheegaard} we established that a random hyperbolic 3-manifold of Heegaard genus $g$ admits a Riemannian metric of sectional curvature close to $-1$ with some specific geometric properties. Furthermore, for any given numbers $b>1,\delta >0$, 
a definitive proportion of the volume for this metric is 
contained in a union of pairwise disjoint linearly aligned 
$(b,\delta)$-product regions. Here 
the proportionality constant depends on the numbers $b,\delta$.

The main goal of this section is to show that this property carries over to the
hyperbolic metric on a random 3-manifold. The following proposition
shows that this suffices for the proof of Theorem \ref{main} from the introduction.

\begin{pro}\label{finite}
For fixed $g\geq 2, \delta>0$ and sufficiently large $b>1$, there exists a
number $C=C(g,b,\delta)>0$ with the following property.
Let $M$ be a hyperbolic 3-manifold, and suppose that 
$M$ contains $n\geq 1$ pairwise disjoint linearly aligned
$(b,\delta)$-product regions of genus $g$. Then 
$\lambda_1(M)\leq C/n^2$ and $\lambda_{n}(M)\leq 1/C$.
\end{pro}
\begin{proof} The argument follows the proof of Proposition 4.4 of \cite{BGH16}.
For completeness, we give a sketch.
 
Let $M$ be as in the proposition. Denote by 
${\mc V}=\cup_{i=1}^nV_i\subset M$ the union of the $n$ 
linearly aligned $(b,\delta)$-product regions of genus $g$ 
whose existence is assumed in the statement of the proposition.
Assume that the components  $V_1,\dots,V_n$ of ${\mc{V}}$ are ordered in such 
a way that for all $i< n$, the components $V_i$ and $V_{i+1}$ are contained 
in the boundary of the same component of $M-{\mc{V}}$. 
By Proposition 2.3 of \cite{BGH16} (which is local and hence whose proof
carries over to the situation at hand), we may assume that the boundaries of the
components $V_i$ are smooth. Furthermore, we may assume that there exists
a number $L>0$ only depending on $b$, and for each $i$ there exists a 
diffeomorphism $\psi_i:V_i\to \Sigma\times [0,1]$ which is $L$-Lipschitz. 
We choose this diffeomorphism in such a way that it maps the boundary component
of $V_i$ which is shared with the component of $M-V_i$ containing $V_{i-1}$ to 
$\Sigma\times \{0\}$. 

For each $i$ there exists a smooth function
\[f_i:V_i\to [i-1,i]\]
of uniformly bounded derivative which maps 
$\psi_i^{-1}(\Sigma\times \{0\})$ to $i-1$, and maps  $\psi_i^{-1}(\Sigma\times \{1\})$ to $i$. 
Define a function $f:M\to [0,n]$ by $f\vert V_i=f_i$ and by the requirement that
$f$ is constant on each of the components of $M-{\mc{V}}$.

Define functions $\alpha,\beta$ on $[0,n]$ by 
\[\alpha(s)=\begin{cases} \sin(2\pi s/n), & \text{ if }0\leq s\leq n/2\\
0 & \text{ if }n/2\leq s\leq n\end{cases} \]
and 
\[\beta(s)=\begin{cases} 0, & \text{ if }0\leq s\leq n/2\\
\sin(2\pi (s-n/2)/n),  & \text{ if }n/2\leq s\leq n.\end{cases} \]

Then $\alpha\circ f,\beta\circ f$ are smooth, with supports intersecting in a zero volume set, and their
Rayleigh quotients are uniformly equivalent to $1/n^2$. Namely, the Rayleigh quotients of 
$\alpha,\beta$ are $\pi^2/n^2$, and since $f$ has uniformly bounded derivative, the Ralyeigh quotients
of $\alpha,\beta$ are uniformly equivalent to the Rayleigh quotients of $\alpha\circ f,\beta \circ f$.

By the Minmax theorem for the spectrum of the Laplacian, we know that for any set of functions
$\rho_0,\dots,\rho_k:M\to \mathbb{R}$ whose supports pairwise intersect on zero-volume sets, we have
$\lambda_k\leq \max\{{\mc{R}}(\rho_i)\mid 0\leq i\leq k\}$ and therefore 
$\lambda_1(M)\leq \max\{{\mc{R}}(\alpha\circ f),{\mc{R}}(\beta\circ f)\}$.
From this we conclude that $\lambda_1(M)\leq d/n^2$ where $d>0$ is a universal constant. 

The same argument can be applied to the functions 
\[\rho_i(x)=\begin{cases} \sin(\pi f_i(x)), & \text{ if }x\in V_i\\
0  &\text{ otherwise } \end{cases}\]
whose Rayleigh quotient is uniformly bounded and whose supports are pairwise disjoint.
This yields that $\lambda_n(M)\leq c$ where $c>0$ is a universal constant.
\end{proof}

Theorem \ref{main} from the introduction now follows from Proposition \ref{random relative bounded},
Lem\-ma \ref{finite} and the following statement which is the main result of this section. 
Recall that by hyperbolization, a closed $3$-manifold $M$ which admits a Riemannian metric
of sectional curvature contained in $[-1-\epsilon,-1+\epsilon]$ for some $\epsilon <1/2$ admits a hyperbolic
metric, unique up to isometry by Mostow rigidity.

\begin{thm}\label{bcg} 
For every $g\geq 2,a\in (0,1),b>4,\delta>0$ there exist numbers 
$\epsilon=\epsilon(g,a,b,\delta)>0,a^\prime=a^\prime(g,a,b,\delta)\in (0,1)$ with the following
property. Let $M$ be a closed aspherical atoroidal 
$3$-manifold of Heegaard genus $g$, and
let $\rho$ be 
a Riemannian metric on $M$ of curvature contained in $(-1-\epsilon,-1+\epsilon)$.
Assume that $(M,\rho)$ contains a linearly aligned collection ${\mc V}$ of pairwise disjoint
$(b,\delta)$-product regions of genus $g$ whose total volume is at least 
$a {\rm vol}(M,\rho)$. Let $\rho_0$ be the hyperbolic metric on $M$.
Then $(M,\rho_0)$ contains a linearly aligned collection  
${\mc W}$ of pairwise disjoint $(b-1,\delta/2)$-product regions 
of volume at least $a^\prime{\rm vol}(M,\rho_0)$.
 \end{thm}

By Proposition \ref{random relative bounded}, for a fixed choice of a number
$b>4$ and sufficiently small $\delta >0$, a random 3-manifold $M$ of Heegaard genus $g$
admits a Riemannian metric $\rho$ which 
fulfills the assumption in Theorem \ref{bcg} for some number $a\in (0,1)$. Note that
$b,\delta$ are independent of $M$, and the number $a\in (0,1)$ depends on the random walk. 
Thus Theorem \ref{main} is an immediate consequence of Theorem \ref{bcg}
and Proposition \ref{finite}. 

We are left with the proof of Theorem \ref{bcg} which is carried out in the 
remainder of this section.
We use a construction of \cite{BCG95}, \cite{BCG99}. 
The following is a special case of the main result of \cite{BCG99}.

\begin{thm}\label{bcg2}
Let $(M,\rho)$ and $(M_0,\rho_0)$ be closed oriented Riemannian manifolds
of dimension $3$ and suppose that for some constant $b\geq 1$ 
\[{\rm Ric}_\rho\geq -2, \text{ and }\quad  -b^2\leq K_{\rho_0}\leq -1.\]
If there exists a map $f:M\to M_0$ of degree one then
\[{\rm vol}(M,\rho)\geq {\rm vol}(M_0,\rho_0),\]
with equality if and only if $(M,\rho),(M_0,\rho_0)$ are isometric and hyperbolic. 
\end{thm}

Here ${\rm Ric}_\rho$ and $K_{\rho_0}$ are the Ricci curvature and the sectional curvature
of $\rho$ and $\rho_0$.

\begin{cor}\label{volumecomparison}
For $\epsilon <1/2$ let 
$\rho$ be a Riemannian metric on the closed 3-manifold $M$ of curvature
contained in $(-1-\epsilon,-1+\epsilon)$ and let $\rho_0$ be the hyperbolic metric on $M$. 
Then
\[{\rm vol}(M,\rho)/{\rm vol}(M,\rho_0)\in [(1-\epsilon)^{3/2},(1+\epsilon)^{3/2}].\]
\end{cor}
\begin{proof} Rescaling the metric $\rho$
with the factor $(1-\epsilon)^{-1}$ yields a new metric on $M$ whose volume
is $(1-\epsilon)^{-3/2}{\rm vol}(M,\rho)$ and whose sectional curvature is bounded from 
below by $-1$. In particular,
the Ricci curvature of this metric is at least $-2$. 
An application of Theorem \ref{bcg2} then implies that
${\rm vol}(M,\rho)\geq (1-\epsilon)^{3/2}{\rm vol}(M,\rho_0)$.

Similarly, rescaling the metric $\rho$ on $M$ with the factor
$(1+\epsilon)^{-1}$ yields 
a metric whose sectional curvature is bounded from
above by $-1$ and whose volume equals $(1+\epsilon)^{-3/2}{\rm vol}(M,\rho)$. Another application of 
Theorem \ref{bcg2}, with the roles of $(M,\rho)$ and $(M,\rho_0)$ 
exchanged, shows that 
${\rm vol}(M,\rho_0)\geq (1+\epsilon)^{-3/2}{\rm vol}(M,\rho)$.
Together the corollary follows.
\end{proof}

The {\em volume entropy} $h(\rho)$ of a negatively curved metric $\rho$ on $M$
is the asymptotic growth rate of the volume of balls in its universal covering.
The volume entropy of a hyperbolic metric equals $2$, and the 
volume entropy of a metric whose sectional curvature is bounded from below
by $-b^2$ for some $b>0$ is at most $2b$. 

For $c>h(\rho)$ there exists a smooth {\em natural map} $F_c:(M,\rho)\to (M,\rho_0)$
\cite{BCG95}. The following statement summarizes some of the results 
from Section 7 of
\cite{BCG95}. Part of the statement is only implicitly contained in \cite{BCG95}, 
but an explicit version can be found in 
Theorem 2.1 of \cite{BCS05}. We always assume that the constant $\epsilon$ which 
controls the curvature of $M$ is smaller than $1/2$ and that the number 
$c>h(\rho)$ is bounded from above by $4$ to make all constants uniform. 

\begin{pro}\label{bcg3}
Let $c>h(\rho)$ and let $F_c:(M,\rho)\to (M,\rho_0)$ be the natural map. 
\begin{enumerate}
\item $F_c$ is of degree one, and its Jacobian satisfies
\[\vert {\rm Jac}(F_c)\vert \leq \bigl(\frac{c}{2}\bigr)^3\]
pointwise. 
\item There are $\kappa>0,r \in (0,1)$ and $L>1$ not depending on $(M,\rho)$
with the following property. If $x\in (M,\rho)$ is such that 
$\vert {\rm Jac}(F_c)(x)\vert \geq (1-\kappa)(\frac{c}{2})^3$
then the restriction of the map $F_c$ to the ball $B(x,r)$ of radius $r$
about $x$ in $(M,\rho)$ is $L$-Lipschitz.
\item For all $\theta>0$ and $x\in M$ there exists $\beta>0$ such that if 
$\vert {\rm Jac}(F_c)(x)\vert \geq (1-\beta)(\frac{c}{2})^3$ then 
\[(1-\theta)(\frac{c}{2})^3<\vert d_xF_c(v)\vert <(1+\theta)(\frac{c}{2})^3\]
for all unit tangent vectors $v\in T_xM$. 
\end{enumerate}
\end{pro}

The strategy is now as follows. Given $a\in (0,1)$ and 
$b>4L$ where $L>1$ is as in Proposition \ref{bcg3}, 
for a manifold $(M,\rho)$ which fulfills the assumption in 
Theorem \ref{bcg} for 
sufficiently small $\epsilon>0$, we find 
a union ${\mc W}\subset {\mc V}$ of components of the collection
${\mc V}$ of $(b,\delta)$-product regions in $(M,\rho)$ 
whose total measure is large and such that the restriction to this set of the 
natural map $F_c:(M,\rho)\to (M,\rho_0)$ for a suitably chosen $c>h(\rho)$ 
has large Jacobian outside of a subset which does not contain any
ball of radius $r$ 
where $r>0$ is as in the second part of Proposition \ref{bcg3}.
Proposition \ref{bcg3} then yields that the map $F_c$ is uniformly Lipschitz 
on ${\mc W}$.
We then argue that the image under $F_c$ of a $(b,\delta)$-product
region in ${\mc W}$ 
contains a $(b^\prime,\delta^\prime)$-product region in $(M,\rho_0)$
where $b^\prime$ is close to $b$ and $\delta^\prime$ is close to $\delta$.
The geometric control on the image of the map $F_c$ is then used to show that 
suitably chosen sub-regions of these image product regions of controlled total 
volume are pairwise
disjoint and linearly aligned. 

The following lemma establishes a first volume control. In its formulation,
the numbers $r>0,L>1$ are as in Proposition \ref{bcg3}.
 
\begin{lem}\label{firstvolume}
Let $a\in (0,1),b>\max\{10r,4\},\delta >0$ and $\xi>0$.
There exists a number $\epsilon_0=\epsilon_0(a,b,\delta,\beta)>0$ 
with the following property. Let
$(M,\rho)$ be as in Theorem \ref{bcg}, with sectional curvature
contained in $(-1-\epsilon_0,-1+\epsilon_0)$.  Then for $c>h(\rho)$ sufficiently
close to $h(\rho)$, there is a subset ${\mc W}\subset {\mc V}$ with the following properties.
\begin{enumerate}
\item ${\mc W}$ 
is a union of components of ${\mc V}$, and its total volume is at least 
$a{\rm vol}(M,\rho)/2$. 
\item 
The restriction of $F_c$ to each component of ${\mc W}$ is 
$L$-Lipschitz, and its image is contained in the $\sigma$-thick part of 
$(M,\rho_0)$ for a universal constant $\sigma >0$.
\item If $V$ is any component of ${\mc W}$ then ${\rm vol}(F_c(V))\geq 
(1-\xi){\rm vol}(V)$, and there exists a subset $A$ of $V$ with
${\rm vol}( A)\geq a{\rm vol}(V)$ such that 
$F_c^{-1}(F_c(x))\subset V$ for all $x\in A$. 
\end{enumerate}
\end{lem}
\begin{proof}
Let $r>0$ be as in the second part of Proposition \ref{bcg3}. Assume without loss
of generality that $r<1$. 
For $x\in (M,\rho)$ let $B(x,r)$ be the open ball of radius $r$ about $x$. 
Let ${\mc V}$
be a union of $(b,\delta)$-product regions as in the statement of 
Theorem \ref{bcg}. 
Since the components of ${\mc V}$ are linearly aligned and $b>4$, 
any ball $B(y,r)$ in $(M,\rho)$ intersects at most two different
components of ${\mc V}$.

Let us consider a point $x\in V$. The injectivity radius of $(M,\rho)$
at $x$ is at least 
$\delta$. Therefore by comparison, the volume of the ball $B(x,r)$ 
is bounded from below by a universal constant $\alpha>0$. 
On the other hand, as the diameters of the boundary surfaces
of a component $V$ of ${\mc V}$ are uniformly bounded, 
the volume of the $r$-neighborhood 
$N_r(V)$ of 
any component $V$ of ${\mc V}$ is bounded from above by a universal constant
$\beta >0$.
Thus if $x\in V$ then the ratio ${\rm vol}(B(x,r))/{\rm vol}(N_r(V))$ is bounded from 
below by a universal constant $\alpha/\beta$. 

Let $\xi>0$. Define 
\[Z=\{x\in M\mid \vert {\rm Jac}(F_c)(x)\vert \geq (1-\xi)(\frac{c}{2})^3\}.\]
By Corollary \ref{volumecomparison} and 
the first part of Proposition \ref{bcg3}, for sufficiently small $\epsilon>0$ and for 
$c>h(\rho)$ sufficiently close to $h(\rho)$, 
the volume of the union ${\mc W}$ of all components $V$ 
of ${\mc V}$ with the property that 
$N_r(V)-Z$ does not contain a ball of radius $r$ centered at a point $x\in V$ 
is at least $3a{\rm vol}(M)/4$.
Namely, if $V_1,\dots,V_k$ are the components of 
${\mc V}-{\mc W}$ and if $x_i\in V_i$ is such that $B(x_i,r)\subset M-Z$, then 
by the above discussion, any of the balls $B(x_i,r)$ intersects at most 
one other ball $B(x_j,r)$ for $j\not=1$. In particular, at 
least $k/2$ of the balls $B(x_i,r)$ are pairwise
disjoint and hence 
\[{\rm vol}(\cup_iB(x_i,r))\geq k\alpha/2.\] 
Thus if ${\rm vol}({\mc V}-{\mc W})\geq a{\rm vol}(M,\rho)/4$
then 
${\rm vol}(\cup_iB(x_i,r))\geq \alpha a{\rm vol}(M,\rho)/8\beta$. 
But the restriction of $F_c$ to 
$\cup_iB(x_i,r)$ decreases the volume by a definitive factor. For 
$\epsilon >0$ sufficiently close to $0$ and $c-h(\rho)>0$ sufficiently small,
this violates Corollary \ref{volumecomparison}.

By the second part of Proposition \ref{bcg3}, the restriction of $F_c$ to any 
component $V$ of ${\mc W}$ is $L$-Lipschitz where $L>1$ is a universal 
constant. In particular, if 
$\gamma$ is a closed loop entirely contained in 
$V$, then the length of its image $F_c(\gamma)$
is at most $L$ times the length of $\gamma$.

By the definition of a $(b,\delta)$-product region, 
for an arbitrary point $x\in V$ the subgroup of 
$\pi_1(M)$ generated by the homotopy classes of uniformly
short loops at 
$x$ which are entirely contained in $V$ 
is not virtually abelian. But this implies that for any point 
 $y\in F_c(V)$, there are closed loops of uniformly bounded length
 passing though $y$ which generated a non-solvable subgroup of 
$\pi_1(M)$. As a consequence, the set $F_c(V)$ is 
contained in the
$\sigma$-thick part of $(M,\rho_0)$ for a universal constant $\sigma >0$. 
Together this shows the first and second part of the lemma.

Now if $V$ is a component of ${\mc W}$ and if 
$B=\{x\in V\mid \vert F_c^{-1}(F_c(x))\not\subset V\}$ then 
the volume of $(M,\rho_0)$ equals the volume of 
$F_c(M-B)$. Thus 
as $\epsilon\to 0$ and $c-h(\rho)\to 0$,
by volume comparison the proportion of the 
volume of ${\mc W}$ contained in the union of those components of 
${\mc W}$ 
which violate the conditions in the third part of the lemma has to 
tend to zero. This then implies the third part of the lemma.
\end{proof}

For a number $\xi>0$ we say that a map $F$ between two metric spaces
$X,Y$ is a {\em $\xi$-coarse isometry}
if $\vert d(Fx,Fy)-d(x,y)\vert \leq \xi$ for all $x,y$. 

\begin{lem}\label{nextcompare}
For $b^\prime<b,\delta^\prime<\delta$ and $\xi>0$ 
there exists a number
$\epsilon_0=\epsilon_0(b^\prime,\delta^\prime)$ with the following property.
Let $(M,\rho)$ be as in Lemma \ref{firstvolume}
and let $V$ be  a component of ${\mc W}$ where ${\mc W}$ is as in Lemma \ref{firstvolume},
then the restriction of $F_c$ to $V$ is a $\xi$-coarse isometry whose image
contains a $(b^\prime,\delta^\prime)$-product region of genus $g$.
\end{lem}
\begin{proof}
We argue by contradiction and we assume that a number $\epsilon_0>0$ 
as in the lemma does not
exist. Then there exists a sequence
of closed $3$-manifolds $(M_i,\rho)$ which fulfill the assumptions in 
Theorem \ref{bcg} for a
sequence $\epsilon_i\to 0$ and fixed numbers $g\geq 2,a>0, b>4,\delta>0$ and such that
for each $i$, there is a component $V_i$ of the collection 
${\mc W}_i$ as in Lemma \ref{firstvolume} whose image under the 
natural map $F_i:(M_i,\rho)\to (M_i,\rho_0)$ does not contain a 
$(b^\prime,\delta^\prime)$ product region where $b^\prime<b$ and $\delta^\prime<\delta$ are fixed
constants. 
Note that in contrast to similar statements in the literature, we do not assume
the existence of a bound on 
the diameters of the manifolds $(M_i,\rho)$. Let as before $\rho_0$ be the 
hyperbolic metric on the manifold $M_i$.

Let $h_i$ be the volume entropy of $M_i$. We know that 
$h_i\to 2$ $(i\to \infty)$. 
Choose a sequence $\chi_i\to 0$ such that 
$h_i<2+\chi_i$. 
For each $i$ consider the natural map
$F_i:(M_i,\rho)\to (M_i,\rho_0)$ for the parameter $c_i=2+\chi_i$. 
By the choice of ${\mc W}_i$ and the second part of Lemma \ref{firstvolume}, 
we know that the restriction of $F_i$ to $V_i$ is $L$-Lipschitz where
$L>1$ does not depend on $i$. Furthermore, 
for each $\beta>0$, the measure of the set of all points $z\in V_i$ so that
$\vert {\rm Jac}(F_i)(x)\vert \leq (1-\beta)(\frac{c_i}{2})^3$ tends to zero as $i\to \infty$.
By the third part of Proposition \ref{bcg3}, as $i\to \infty$, 
on a subset of the component $V_i$ of ${\mc W}_i$ containing 
a larger and larger proportion of the volume of
$V_i$, the differential of $F_i$ is close to an isometry. 

For each $i$ let $x_i\in V_i$. The set $F_i(V_i)$ is contained in the
$\sigma$-thick part of $(M_i,\rho_0)$ where $\sigma$ does not depend on $i$.
Thus by passing to a subsequence, we may assume  
that the pointed manifolds $(M_i,x_i,\rho)$ converge 
in the geometric topology to a pointed hyperbolic manifold $(M,x)$ and that the 
pointed hyperbolic manifolds $(M_i,F_i(x_i),\rho_0)$ converge in the geometric
topology to a pointed hyperbolic manifold $(N,y)$. 

Let $(V,x)$ be the geometric limit of the pointed
$(b,\delta)$-product regions $(V_i,x_i)$.
Then $V$ is a $(b,\delta)$-product region in $M$ containing the basepoint $x$.
Furthermore, as the restriction of $F_i$ to 
$V_i$ is $L$-Lipschitz for a universal constant $L>1$, up to passing 
to another subsequence we may assume that $F_i\vert V_i$ converges
to an $L$-Lipschitz map $F:(V,x)\to (N,y)$.

By the definition of geometric convergence, 
for large enough $i$ there exists a $(1+\xi_i)$-bilipschitz homeomorphism 
$\phi_i$ of a neighborhood 
$U$ of $V$ in $M$ onto a neighborhood $U_i$ of $V_i$ in $M_i$ where
$\xi_i\to 0$ $(i\to \infty)$. We use $\phi_i$ to identify $U$ with $U_i$. 

As $i\to \infty$ and by the choice of the sets $V_i$, 
the Jacobians of the restriction of $F_i$ to 
$V_i$ converge to one almost surely. We now follow the reasoning 
in the proof of Lemma 7.5 of \cite{BCG95}. Namely, 
using the map $\phi_i^{-1}$ we can think of $U_i$ as a neighborhood of 
$V$ in $M$. 
Egoroff's theorem then implies that for each $n$ 
there exists a subset $K_n\subset V$ with ${\rm vol}(V-K_n)<1/n$ and such that
on $K_n$ 
the differentials $dF_i$ converge to an isometry 
uniformly. By Lemma 7.7 and Lemma 7.8 of \cite{BCG95}, 
the map $F\vert V$ is one-Lipschitz.
Its differential exists almost everywhere and is an isometry. 
It then follows from Appendix B that $F:V\to N$ is an isometric embedding.
In particular, $F(V)$ is a $(h,\delta)$-product region in $N$, and for sufficiently
large $i$ the map $F_i$ is a $\xi$-coarse isometry. 

Geometric convergence now implies that for large enough $i$, 
the image of $V_i$ under $F_i$ is a $(b^\prime,\delta^\prime)$-product
region in $(M_i,\rho_0)$. This is a contradiction to the assumption on 
the sets $V_i$. 
\end{proof}

\begin{proof}[Proof of Theorem \ref{bcg}]
We showed so far that 
for sufficiently small $\epsilon_0>0$, if $(M,\rho)$ is as in Theorem \ref{bcg},
of sectional curvature contained in $(1-\epsilon_0,1+\epsilon_0)$,
then $(M,\rho_0)$ contains a union 
of $(b^\prime,\delta^\prime)$-product
region for some $b^\prime$ close to $b$, $\delta^\prime$ close to $\delta$
which cover a fixed proportion of the volume of $(M,\rho)$.
These product regions are the images under a suitably chosen
natural map $F_c$ of a subcollection ${\mc W}\subset {\mc V}$ 
of the family ${\mc V}$ of $(b,\delta)$-product regions whose existence
is assumed for $(M,\rho)$. Furthermore, the volume of ${\mc W}$ 
is at least $a{\rm vol}(M,\rho)$ for some fixed number $a>0$ (with a slight
abuse of notation). The restriction of $F_c$ to ${\mc W}$ is $L$-Lipschitz
and a $1/4$-coarse isometry, 
and ${\rm vol}(F_c({\mc W}))/{\rm vol}({\mc W})$ is very close to one.

Let $\hat b<b-2$ and $\hat \delta <\delta$ be such that each component 
$V$ of 
${\mc W}$ contains a $(\hat b,\hat \delta)$-product region 
$\hat V$ in its interior whose
one-neighborhood is entirely contained in $V$. The volume
of $\hat V$ is at least $b{\rm vol}(V)$ for a universal constant $b>0$.

Our goal is to show that there is a subcollection ${\mc Z}$ of ${\mc W}$ of 
volume at least 
$a{\rm vol}(M,\rho)/2$ with the additional property that whenever
$V\not=W\in {\mc Z}$ then $F_c(\hat V)\cap F_c(\hat W)=\emptyset$.

To this end let us assume that for $V\not=W\in {\mc W}$ we have  
$F_c(\hat V)\cap F_c(\hat W)\not=\emptyset$. As the restriction of the map 
$F_c$ is $L$-Lipschitz and a $1/4$-coarse isometry, this implies that 
there are balls $B_1\subset V,B_2\subset W$
of radius $1/2L$ such that $F_c(B_1)\subset F_c(W)$ and 
$F_c(B_2)\subset F_c(V)$. Namely, for all $z\in \hat V$ the ball of radius $1/2$ about
$F_c(z)$ is contained in $F_c(V)$, furthermore $F_c$ is $L$-Lipschitz.

Let $2\sigma>0$ be a lower bound for the volume of a ball of radius 
$1/2L$ entirely contained in an $(b,\delta)$-product region. Such a number exists
since the injectivity radius in such a region is at least $\delta$.
Then the volume of $F_c(V\cup W)$ is at most 
$(\frac{c}{2})^3 ({\rm vol}(V)+{\rm vol}(W)-2\sigma)$. 
In particular, the contribution of $F_c(V)$ to the volume of ${\mc W}$ does not exceed
$(\frac{c}{2})^3({\rm vol}(V)-\sigma)$. 

Since $\sigma>0$ is independent of all choices and for $c$ sufficiently close to $2$
the restriction of the map $F_c$ to ${\mc W}$ is very close to being 
volume preserving, 
we deduce that for $c$ sufficiently close to $2$ the union ${\mc Z}$ of all product regions
$\hat V$ with $V\in {\mc W}$ and such that the sets from ${\mc Z}$ are mapped disjointly
by $F_c$ covers a fixed proportion of the volume of $(M,\rho_0)$. 
Furthermore, the image of each of the components in ${\mc Z}$ contains 
a $(b^{\prime},\delta^{\prime})$-product region for some fixed 
$b^\prime<\hat b$ and some $\delta^\prime$ close to $\hat \delta$. Thus we found
a collection of pairwise disjoint product regions in $(M,\rho_0)$ as claimed in the
theorem. 

We are left with showing that the regions $F_c(\hat V)$ for $\hat V\in {\mc Z}$ are 
linearly aligned. 
However, $F_c$ is a homotopy equivalence. If $\hat V\in {\mc Z}$ then 
as the restriction of $F_c$ to $\hat V$ is a homeomorphism, for a fixed
choice of an embedded surface $\Sigma\subset V$ which decomposes
$M$ into two handlebodies, the image surface $F_c(\Sigma)$ separates
$(M,\rho_0)$ into two components. The restriction of $F_c$ to the closure of 
a component of $M-\Sigma$ is a generator of the relative homology group
$H_3(M,M-F_c(\Sigma))$. But this homology group also is generated
by the inclusion of a component of $M-F_c(\Sigma)$ and hence each 
component $A$ of $M-\Sigma$ determines uniquely a component 
${\mc F}(A)$ of 
$M-F_c(\Sigma)$ with the additional property that 
$F_c(A)\supset {\mc F}(A)$.

Now let $\hat V\not=\hat W\in {\mc Z}$. As the components of 
${\mc Z}$ are pairwise disjoint, the component $\hat W$ is entirely contained
in a component of $M-\hat V$, say the component $A$.
Furthermore, as 
$F_c(\hat V),F_c(\hat W)$ are disjoint, the component $F_c(\hat W)$ is contained in 
a component $Z$ of $M-F_c(\hat V)$. We claim that 
$Z={\mc F}(A)$. 

Namely, let $B$ be the component of $M-\hat W$ entirely contained in $A$.
If $Z\not={\mc F}(A)$ then we have 
$F_c(\hat V)\subset {\mc F}(B)$. But the restriction of $F_c$ to $B$ maps
$B$ to a subset that contains ${\mc F}(B)$. In particular, we have
$F_c(\hat V)\subset F_c(M-\hat V)$ which violates property (3) 
in Lemma \ref{firstvolume}.

But this just means that the components of $F_c({\mc W})$ are linearly aligned.
This completes 
the proof of the theorem.
\end{proof}

\appendix

\section{Local control of one-Lipschitz maps}\label{local}

The goal of this appendix is to show:

\begin{pro}\label{localprop}
Let $U$ be a domain in a hyperbolic 3-manifold and let $F:U\to N$  be a 
volume preserving 1-Lipschitz
map into a hyperbolic 3-manifold $N$. Then $F$ is an isometric embedding.
\end{pro}

Compare Appendix C of \cite{BCG95} for a different variation

\begin{proof}
As $F$ is volume preserving, all we need to show that $F$ is a local isometry. To this end let $x\in U$ and let $r_0>0$ be such that the closed balls 
$B(x,r_0)$, $B(F(x),r_0)$ 
of radius $r_0$ 
about $x$ and $F(x)$ are isometric to the closed ball of the same radius in hyperbolic 3-space. 
Since $F$ is one-Lipschitz we know that $F(B(x,r_0))\subset 
B(F(x),r_0)$. Furthermore, as $F$ is continuous and $B(x,r_0)$ is compact, 
$F(B(x,r_0))$ is a closed subset of $B(F(x),r_0)$ and hence coincides with 
$B(F(x),r_0)$ as $F$ is volume preserving.

Since $F$ is one-Lipschitz, it is differentiable almost everywhere, and its differential
is norm non-increasing. Since
$F$ is moreover volume preserving, the differential of 
$F$ is an isometry almost everywhere. 
Furthermore, the set of all points $x\in U$ such that $F^{-1}(F(x))=\{x\}$ has full measure.

Let $x$ be such a point. We saw above that there is a closed subset $A$ of 
the distance sphere $A$ of radius $r_0$ about $x$ is mapped by $F$ onto the distance
sphere of radius $r_0$ about $F(x)$. Note that we do not know at this point whether $A$ equals
the entire distance sphere as we do not know whether $F$ is injective- there could a priori be
points in the distance of radius $r_0$ about $x$ which are mapped to the interior of the ball 
$B(F(x),r_0)$. 
If $y\in A$ then using once more that $F$ is 
a contraction, the geodesic $\gamma_y$ connecting $x$ to $y$ is mapped by 
$F$ to the geodesic $\gamma_{Fy}$ connecting $F(x)$ to $F(y)$. As $F$ is differentiable at 
$x$ and $dF(x)$ is an isometry, we have 
$dF(\gamma_y^\prime(0))=\gamma_{Fy}^\prime(0)$. In particular, if
$\exp$ denotes the exponential map at $x$ then $F(\exp(s\exp^{-1}(z)))=
\exp(sdF(\exp^{-1}(z))$ for all $z\in A$. On the other hand, $F(A)=\partial B(F(y),r_0)$ and
hence $A=\partial B(x,r_0)$ and the restriction of $F$ to $B(x,r_0)$ is an isometry.

As $x$ was a point from a subset of $U$ of full measure, $F$ is indeed a local isometry
and hence an isometry.
\end{proof}

\noindent
Mathematisches Institut der Universit\"at Bonn,\\
Endenicher Allee 60, 53115 Bonn, Germany\\
E-mail: ursula@math.uni-bonn.de

\bigskip

\noindent
Mathematisches Institut der Universit\"at Heidelberg,\\
Im Neuenheimer Feld 205, 69120 Heidelberg, Germany\\
E-mail: gviaggi@mathi.uni-heidelberg.de

\end{document}